\documentclass[12pt, reqno]{amsart}
\usepackage{amsmath, amssymb, amscd, amsthm, comment, amsxtra,mathrsfs, enumerate}
\usepackage{graphicx}
\usepackage[utf8]{inputenc}
\usepackage[T1]{fontenc} 
\usepackage{subfigure}
\usepackage[pdftex, linktocpage=true]{hyperref}
\usepackage{euscript}
\usepackage[format=plain,labelfont=bf,up,width=.99\textwidth]{caption}
\usepackage[toc,page]{appendix}
\usepackage[all]{xy}%
\usepackage[swedish,english]{babel}
\usepackage{xcolor}

\theoremstyle{plain}
\newtheorem{thm}{Theorem}[section]
\newtheorem{lem}[thm]{Lemma}
\newtheorem{prop}[thm]{Proposition}
\newtheorem{cor}[thm]{Corollary}

\newtheorem*{claim*}{Claim}
\newtheorem*{theorem*}{Main}
\theoremstyle{definition}
\newtheorem{defn}[thm]{Definition}

\newtheorem{rem}[thm]{Remark}

\newtheorem{assum}[thm]{Assumption}

\newtheorem*{conj*}{Conjecture}
\newtheorem*{rem*}{Remark}
\newtheorem{notn}[thm]{Notation}

\numberwithin{equation}{section}
\newcommand{\thmref}[1]{Theorem~\ref{#1}}
\newcommand{\propref}[1]{Proposition~\ref{#1}}

\newcommand{\lemref}[1]{Lemma~\ref{#1}}
\newcommand{\corref}[1]{Corollary~\ref{#1}}

\newcommand{\C}[1]{{\Bbb C_{#1}}}
\newcommand{\I}{P}

\newcommand{\cA}{{\cal A}}

\newcommand{\cU}{{\mathcal U}}

\newcommand{\cW}{{\cal W}}

\newcommand{\cV}{{\cal V}}
\newcommand{\cF}{{\cal F}}

\newcommand{\cB}{{\cal B}}
\newcommand{\cT}{{\cal T}}

\newcommand{\cP}{{\cal P}}
\newcommand{\cC}{{\cal C}}

\newcommand{\cR}{{\cal R}}
\newcommand{\cL}{{\cal L}}

\newcommand{\cE}{{\cal E}}
\newcommand{\cS}{{\cal S}}

\newcommand{\cK}{{\cal K}}

\renewcommand{\deg}{\operatorname{deg}}

\newcommand{\matsp}[1]{\hspace{5mm} \text{#1} \hspace{5mm}}

\def\cal{\mathcal}

\def\C{{\mathbb C}}
\def\D{{\mathbb D}}

\def\I{{\mathcal I}}

\def\Q{{\mathbb Q}}
\def\R{{{\mathbb R}}}

 \def\U{{\mathcal U}}
\def\V{{\mathcal V}}

\def\Z{{\mathbb Z}}
\def\pB{{\mathcal B}}
 
 \def\epsilon{{\varepsilon}}

\newcommand{\bbC}{\mathbb C}
\newcommand{\bbD}{\mathbb D}

\newcommand{\bbN}{\mathbb N}

\newcommand{\bbQ}{\mathbb Q}
\newcommand{\bbR}{\mathbb R}

\newcommand{\bbZ}{\mathbb Z}

\newcommand{\bfB}{\mathbf B}

\newcommand{\bfH}{\mathbf H}

\newcommand{\bfP}{\mathbf P}

\newcommand{\bfT}{\mathbf T}

\newcommand{\rt}{\operatorname{root}}

\newcommand{\gen}{\operatorname{gen}}

\newcommand{{\q}}{q_{\theta}}
\newcommand{\bB}{{{\bf B}}}

\newcommand{\bR}{R^{\bB}}
\newcommand{{\Zak}}{\mathcal{Z}_{\theta}}

\def\b{{b}}

\newcommand{\cZ}{{\cal Z}}
\newcommand{\hc}{\hat c}

\newcommand{\Imp}{\operatorname{Imp}}
\newcommand{\Arc}{\operatorname{Arc}}
\newcommand{\fun}{\operatorname{fun}}

\newcommand{\jy}[1]{{\noindent\color{blue}  [#1]}}
\newcommand{\rz}[1]{{\noindent\color{red}  [#1]}}

\title{Rigidity of bounded type cubic Siegel polynomials}
\author{Jonguk Yang, Runze Zhang}
\begin{document}

\begin{abstract}
We prove that if two non-renormalizable cubic Siegel polynomials with bounded type rotation numbers are combinatorially equivalent, then they are also conformally equivalent. As a consequence, we show that in the one-parameter slice of cubic Siegel polynomials considered by Zakeri \cite{Za2}, the locus of non-renormalizable maps is homeomorphic to a double-copy of a quadratic Siegel filled Julia set (minus the Siegel disk) glued along the Siegel boundary. This verifies the the conjecture of Blokh-Oversteegen-Ptacek-Timorin \cite{BlOvPtTi} for bounded type rotation numbers.
\end{abstract}
\maketitle
\section{Introduction}

Two dynamical systems are said to be {\it combinatorially equivalent} if they are both semi-conjugate to the same topological model. For certain classes of maps, this weak notion of equivalence is enough to guarantee a much stronger relation between the two systems: namely, the existence of a conjugacy with good regularity (ideally, as good as the maps themselves). This phenomenon is referred to as {\it (combinatorial) rigidity}.

Rigidity is a central theme in holomorphic dynamics. The most famous longstanding conjecture in the field is that rigidity holds in the class of quadratic polynomials. Since this implies the local connectivity of the Mandelbrot set (see \cite[Prop. 4.2]{Ly2}), the conjecture is often referred to as the MLC conjecture (which stands for ``Mandelbrot set is Locally Connected'').

The first major breakthrough in this direction was due to Yoccoz (cf. \cite{Hu}). He proved that {\it non-renormalizable} quadratic polynomials (see Definition \ref{defn.ren}) are combinatorially rigid, and consequently, that the Mandelbrot set is locally connected at the corresponding parameters. His work pioneered the use of {\it puzzles}: a sequence of partitions in the dynamical and parameter spaces that reveal the underlying combinatorial structure of the dynamics.

One of the most difficult challenges to overcome when studying the problem of rigidity is the presence of irrationally indifferent periodic orbits. Let $f : \C \to \C$ be a polynomial of degree $\deg(f) \geq 2$ with a fixed point at $0$. Denote its Julia set and Fatou set by $J(f)$ and $F(f) := \C \setminus J(f)$ respectively. We say that $0$ is {\it irrationally indifferent} if:
$$
f'(0) = e^{i 2\pi \theta}
\matsp{for some}
\theta \in (\R \setminus \Q)/\Z.
$$
The angle $\theta$ is referred to as the {\it rotation number}. We say that $0$ is a {\it Siegel point} if $f$ is locally linearizable near $0$ (or equivalently, if $0 \in F(f)$). Otherwise, $0$ is called a {\it Cremer point}. In the former case, the linearizing conjugacy has a maximal extension to a conformal map from the Fatou component $D(f) \subset F(f)$ containing $0$ to the standard disk $\bbD$. The set $D(f)$ is referred to as a {\it Siegel disk}. If $f$ has a Siegel fixed point at $0$, it is referred to as a {\it Siegel polynomial (with rotation number $\theta$)}.

The nature of an irrationally indifferent point depends strongly on the arithmetic properties of its rotation number. Let $p_n/q_n$ be the $n$th convergent of the continued fraction of $\theta$. We say that $\theta$ is {\it Diophantine (of order $k \geq 2$)} if there exists $C > 0$ such that
$$
q_{n+1} < C q_n^k
\matsp{for all}
n \in \bbN.
$$
If $k=2$, then $\theta$ is said to be {\it of bounded type}. Siegel proved that if $\theta$ is Diophantine, then $0$ is a Siegel point \cite{Si}.

The quadratic Siegel polynomial with rotation number $\theta\in\mathbb{R}/\mathbb{Z}$, after normalization, is given by:
\begin{equation}\label{eq.quadratic}
\q(z) := e^{2\pi i \theta}z\left(1-\frac{1}{2}z\right).
\end{equation}

The first seminal result on Siegel polynomials was due to Douady-Ghys-Herman-Shishikura (see \cite{He}). When $\theta$ is of bounded type, using quasiconformal surgery, they constructed a Blaschke product model of $\q$, and thereby were able to conclude that the Siegel boundary $\partial D(\q)$ is a quasi-circle containing the unique critical point $1$.

In this paper, we are interested in studying Siegel polynomials of higher degree. To this end, consider the one-parameter family $\mathcal{P}^{cm}(\theta)$ of cubic polynomials of the form:
\begin{equation}\label{eq.family.fb}
    f_b(z) := e^{2\pi i \theta} z \left(1 - \frac{1}{2}\left(\b+\frac{1}{\b}\right)z + \frac{1}{3} z^2 \right) \hspace{5mm} \text{for } b \in \mathbb{C}^*.
\end{equation}
Observe that $f_b$ has two critical points $1/b$ and $b$. The superscript $cm$ in $\cP^{cm}(\theta)$ indicates the fact that these critical points are considered {\it with markings}. It is easy to verify that every cubic Siegel polynomial with rotation number $\theta$ is conjugate to $f_b = f_{1/b}$ and $f_{-b} = f_{-1/b}$ (and no other map in $\mathcal{P}^{cm}(\theta)$). In particular, the family $\cP^{cm}(\theta)$ is symmetric under the involutions $b \mapsto 1/b$ and $b\mapsto-b$. The {\it cubic Siegel connectedness locus} is defined as
\begin{equation}
    \mathcal{C}(\theta) := \{b\in \mathbb{C}^* \, | \, J(f_b) \text{ is connected}\}.
\end{equation}

\begin{figure}[ht]
\centering 
\includegraphics[width=0.5\textwidth]{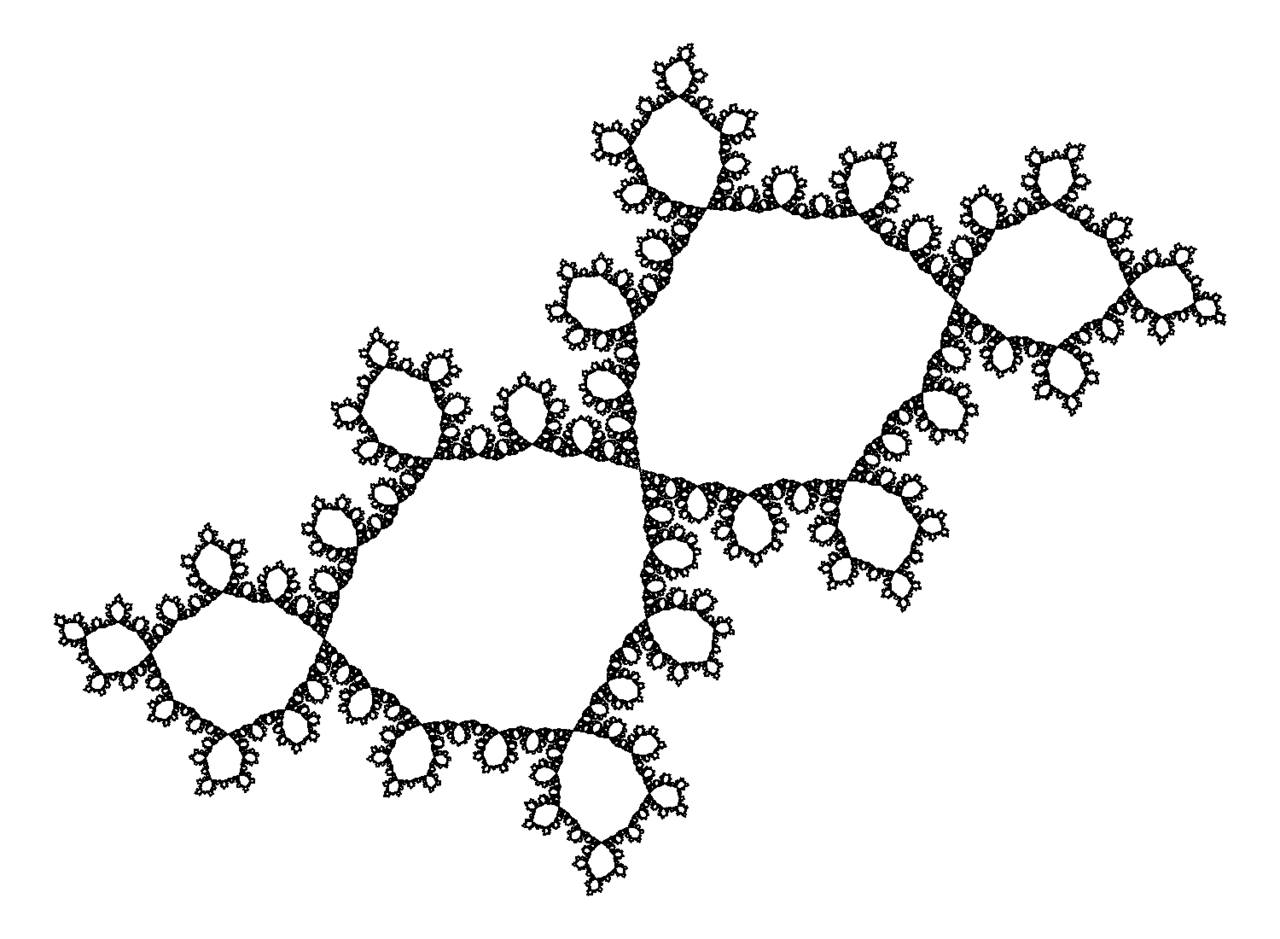} 
\caption{The Julia set $J(\q)$ of the quadratic polynomial $q_\theta(z) = e^{2\pi i\theta}z+ z^2$ with the inverse golden mean rotation number $\theta \approx 0.618$.} 
\label{fig.juliaset} 
\end{figure}

\begin{figure}[ht]
\centering 
\includegraphics[width=0.6\textwidth]{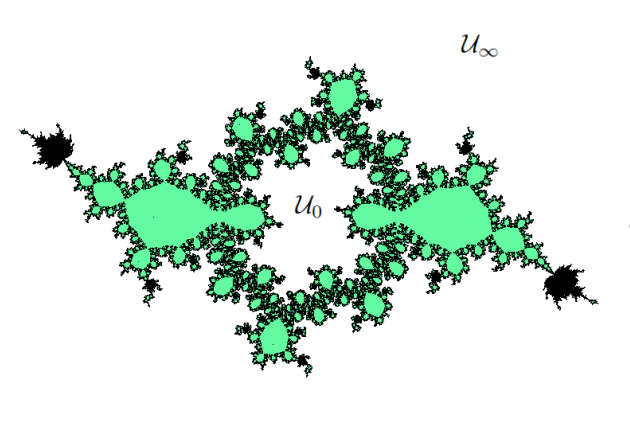} 
\caption{The connected locus $\mathcal{C}(\theta)$ with $\theta=$ the golden mean for the family (\ref{eq.family.fb}). The Zakeri curve is the black quasi-circle that separates $0$ and $\infty$.} 
\label{fig.connectedlocus_b} 
\end{figure}

Now suppose that $\theta$ is of bounded type. Then $0$ becomes a Siegel fixed point with rotation number $\theta$ by a classical result of Siegel \cite{Si}. In \cite{Za2}, Zakeri generalizes the result of Douady et. al. by proving that every $f_b \in \cP^{cm}(\theta)$ has a Blaschke product model (see Proposition \ref{prop.surgery}). Consequently, the Siegel boundary $\partial D(f_b)$ is a quasi-circle which contains at least one critical point. In the quadratic case, this traps the only available critical orbit, and thus, completely fixes the dynamics of the map.  However, for cubic Siegel polynomials, the additional critical point provides a one-dimensional degree of freedom represented by the family $\cP^{cm}(\theta)$. This allows us to investigate how the free critical point interacts with a persistently irrationally indifferent fixed point, and how this interaction is manifested in the topological and geometric properties of the parameter space.

Zakeri also proves that the following set of parameter values:
$$
\Zak := \{c\in \mathbb{C}^* \, | \, b, 1/b \in \partial D(f_b)\}
$$
is a Jordan curve separating $0$ and $\infty$ \cite[Thm. 14.3]{Za2}. We refer to $\Zak$ as the {\it Zakeri curve}. The complement of $ {\Zak}$ consists of two connected components: 
\begin{equation}\label{eq.U0-Uinf}
    \U_0\sqcup \U_\infty := \mathbb{C}^* \setminus {\Zak},
\end{equation}
where $\U_0$ is bounded and $\U_\infty$ is unbounded. For $b\in \U_\infty$, we have $1/b \in \partial D(f_b)$, and $b$ is the free critical point. However, when the value of $b$ crosses $\Zak$ and enters $\U_0$, the role of $b$ and $1/b$ flips: $b$ is now trapped in $\partial D(f_b)$ and $1/b$ becomes the free critical point.

Let us recall the notion of polynomial-like renormalization:
\begin{defn}[\cite{DoHu2}]\label{defn.ren}
    We say that $f_b : \bbC \to \bbC$ is {\it renormalizable}, if there exists two topological disks $U'\Subset U \subset \bbC$ and an integer $p \geq 1$ such that the following properties hold for $g:= f_b^p|_{U'}$:
\begin{enumerate}
\item[$(\mathrm{\romannumeral1})$.] $g:U' \to U$ is a degree two proper map;
\item[$(\mathrm{\romannumeral2})$.] $\bigcap_{n\geq0}g^{-n}(U)$ is connected.
\end{enumerate}
\end{defn}

The {\it central locus $\cK(\theta)$} is the closure of the set of all non-renormalizable cubic Siegel polynomials in $\cP^{cm}(\theta)$. Additionally, let
\begin{equation}\label{eq.central_locus}
    \cK_\infty(\theta) := \cK(\theta) \cap (\cU_\infty\cup \Zak)
\matsp{and}
\cK_0(\theta) := \cK(\theta) \cap (\cU_0\cup \Zak)
\end{equation}
be the {\it external} and {\it internal central locus} respectively. In \S \ref{sec:combina.Zak} and \S \ref{sec:combinatorial.rigidity}, we define a natural notion of combinatorial equivalence between cubic Siegel polynomials. We show that:

\begin{thm}[Rigidity Theorem]\label{thm.main}
Let $\theta \in (\bbR\setminus \bbQ)/\bbZ$ be of bounded type. If $f_b \in \partial \cK(\theta)$, then $f_b$ is combinatorially rigid: if $f_b$ and $f_{b'}$ are combinatorially equivalent, then $f_b$ and $f_{b'}$ are conformally conjugate.
\end{thm}

As a byproduct of the Rigidity Theorem, we obtain the following topological model of $\cK(\theta)$:

\begin{thm}\label{thm.main.siegel}
The external central locus quotient by the involution $b\mapsto-b$ is homeomorphic to the filled Julia set of the quadratic Siegel polynomial $\q$ minus its Siegel disk:
$$
\cK_\infty(\theta)/(b\sim-b) \simeq K(\q) \setminus D(\q).
$$
In particular, $\cK(\theta)$ is locally connected.
\end{thm}

Given a polynomial $f : \bbC \to \bbC$, denote its conformal conjugacy class by $[f]$. Similarly, given a set of polynomials $\cF$, denote
$
[\cF] := \{[f] \; | \; f \in \cF\}.
$
Let
$$
\bfP^3:= \{[f] \; | \; f \; \text{is a cubic polynomial}\}.
$$
The {\it main hyperbolic component}
$\mathbf{H}^3\subset\bfP^3$, which is defined to be the hyperbolic component containing $[z^3]$. In \cite{BlOvPtTi}, Blokh et al. made the following conjecture:
\begin{conj*}
    For all $\theta\in\R/\Z$, $\overline{\bfH^3}\cap[\mathcal{P}^{cm}(\theta)] = [\cK(\theta)].$
\end{conj*}

We are able to verify the conjecture for all bounded type $\theta$:
\begin{cor}\label{cor.main2}
If $\theta$ is of bounded type, then
$
\overline{\bfH^3}\cap[\mathcal{P}^{cm}(\theta)] = [\cK(\theta)].
$
\end{cor}

\subsection{History and strategy of the proof.}
As presented before, the study of the rigidity for non renormalizable (or finitely renormalizable) polynomials somehow begins implicitly with Yoccoz in the 90s.
Using the combinatorial tool called ”puzzles”, he showed that the non-renormalizable
quadratic polynomial $z^2 + c$ without indifferent periodic point is combinatorially rigid (cf. \cite{Hu}). However, the method
of Yoccoz can not be generalised to polynomials of higher degree since the estimate on the modulus of annulus surrounding the critical point relies essentially on the condition degree 2. The breakthroughs are the rigidity for real
polynomials of any degree \cite{KoShvS} and the rigidity for non-renormalizable unicritical polynomials $z^d + c$ \cite{AvKaLySh}, where new combinatorial and analytic tools are
invented. Since then, the rigidity for many other rational map families with more
complicated combinatorics has been studied 
\cite{Ro, QiRoWaYi, Wa, DrSc, Z}.

It is noteworthy that none of the families mentioned above possesses a stable irrationally indifferent fixed point. To our best knowledge, our result provides the
first family of this kind where combinatorial rigidity is justified for non renormalizable parameters. Our proof generally
follows the strategy in \cite{AvKaLySh}. However, in contrast to the unicritical family,
the graphs that we use to construct Yoccoz puzzles consist of bubble rays, which
are entirely contained in the Julia set. A technical point to apply the strategy in
\cite{AvKaLySh} to our situation is that we need to analyse carefully the orbits that accumulate to the graph, especially, to the Siegel disk boundary. For this, we need a more
recent analytic tool: the a priori bounds for puzzle disks \cite{Y} to control the geometry
of the Julia set near the Siegel disk boundary.

\subsection{Outline of the paper.} In \S \ref{sec:quadratic Siegel} and \S \ref{sec:cubic Siegel} we recall respectively some classical results for the quadratic Siegel polynomial and the cubic Siegel family. We construct dynamical bubble rays for the cubic Siegel polynomials in \S \ref{sec:dym bubbleray}. In \S \ref{sec:combina.Zak} we prove the combinatorial rigidity in a special case where the free critical point belongs to a bubble ray. We establish in \S \ref{sec:ext rays} and \ref{sec:para bubble rays} some basic properties of dynamical external rays and parabubble rays. The core part of this paper is \S \ref{sec:combinatorial-tools} and \S \ref{sec:combinatorial.rigidity}, where we use combinatorial tools to prove the combinatorial rigidity in the recurrent and non recurrent case (Theorem \ref{thm.rigidity}). The two key ingredients in the proof are Lemma \ref{lem.bounded.degree.alongD} and Lemma \ref{lem.zero.measure}. We will use them to control the shape distortion of the Julia set near the Siegel disk boundary. In \S \ref{sec:parameterplane}, we finish the proof of Theorem \ref{thm.main}, \ref{thm.main.siegel} and Corollary \ref{cor.main2}.\\
\paragraph{\textbf{Acknowledgements.}} We would like to thank Michael Yampolsky for introducing us
to this project, and for the many helpful suggestions.

We would also like to thank Xiaoguang Wang and Fei Yang for kindly sharing with
us their manuscript \cite{WaYan} during the writing of this paper.

\section{The Quadratic Siegel Polynomial}\label{sec:quadratic Siegel}

Let $\theta \in (0, 1)$ be an irrational angle of bounded type, and consider the quadratic Siegel polynomial $q_\theta$ given in \eqref{eq.quadratic}. The map $\q$ is characterized by having a Siegel fixed point at $0$ with rotation number $\theta$, and its critical point at $1$. Denote the Siegel disc, the Julia set, the filled Julia set, and the attracting basin of infinity for ${{\q}}$ by $D({{\q}})$, $J({{\q}})$, $K({{\q}})$, and $B_\infty({{\q}})$ respectively. Let
$\phi^0_{\q}: D(\q)\to \D$
be the the linearization map on $D(\q)$, normalized by $\phi^0_{\q}(1)=1$. 

A connected component $B$ of $q^{-n}_\theta(D(\q))$ for some $n \geq 0$ is called a {\it bubble}. The {\it generation} of $B$, denoted by gen$(B)$, is defined to be the smallest number $n \geq 0$ such that $q^{n}_\theta(B) = D(\q)$. An iterated preimage $x$ of the critical point $1$ is called a {\it joint}. The {\it generation} of $x$, denoted by gen$(x)$, is defined to be the smallest number $n \geq 0$ such that $q^{n}_\theta(x) = 1$. Every bubble $B$ of generation $n\geq1$ contains a unique joint $x$ of generation $n-1$ in its boundary. We refer to $x$ as the {\it root} of $B$, and denote $\rt(B) := x$.

For $N \in \bbN \cup \{\infty\}$, consider a sequence of bubbles $\{B_i\}_{i=0}^N$ such that $B_0 = D(\q)$, and
$$
\partial B_{i-1} \cap \partial B_i = \{\rt(B_i)\}
\matsp{for}
1 \leq i \leq N.
$$
Then the set
$$
R^\bB := \bigcup \limits_{i=0}^N \overline{B_i}
$$
is called a {\it bubble ray} for ${{\q}}$. We use the notation $R^\bB \sim \{B_i\}_{i=0}^N$ to indicate the sequence of bubbles that make up $R^\bB$.

Let $\bR \sim \{\overline{B_i}\}_{i=0}^\infty$ be an infinite bubble ray. We say that $\bR$ {\it lands} at $z \in J({{\q}})$ if the sequence of closed bubbles $\overline{B_i}$ converges to $z$ as $i \to \infty$ in the Hausdorff topology. We call $z$ the {\it landing point} of ${R}^{\bB}$. An infinite bubble ray is said to be {\it rational} if it is (pre-)periodic. 

The following theorem is due to Petersen \cite{Pe}. Later Yampolsky \cite{Yam} gave an alternative proof using complex a priori bounds.

\begin{thm}\label{quadratic lc}
The Julia set $J({{\q}})$ is locally connected.
\end{thm}

The {\it tree of bubbles} $T(\q) \subset K(\q)$ for $\q$ is defined to be the union of all bubble rays for $\q$. We record the following consequence of \thmref{quadratic lc}.

\begin{cor}\label{cor.landing.quadratic}
Every infinite bubble ray for $\q$ lands. Moreover, we have
$$
K(\q) = T(\q) \sqcup \{\emph{landing points of infinite bubble rays}\}.
$$
\end{cor}

Let $\bR$ be an infinite bubble ray landing at $x \in J(\q)$. When convenient, we will denote $\bR = \bR_t$, where $t \in \bbR/\bbZ$ is the external angle of the unique external ray that lands at $x$. We denote by $\mathscr{T}\subset\mathbb{R}/\mathbb{Z}$ the collection of all the external angles that are realized by some bubble ray. Additionally, if the fact that $\bR_t$ is a bubble ray in the parameter plane for $\q$ needs to be emphasized, we sometimes denote $\bR_t(\q)$. For example, $\bR_0(\q)$ is the unique fixed bubble ray, and $\bR_{1/2}(\q)$ its unique distinct preimage.


\section{The Cubic Siegel Polynomials}\label{sec:cubic Siegel}

Let $\theta \in (0, 1)$ be an irrational angle of bounded type, and consider a cubic Siegel polynomial $f_b \in \cP^{cm}(\theta)$ given in \eqref{eq.family.fb}. Since $f_b = f_{1/b}$, we may assume without loss of generality that $b \in \overline{\U_\infty}$ (see \ref{eq.U0-Uinf}).

The map $f_b$ is characterized by having a Siegel fixed point at $0$ with rotation number $\theta$, and its critical points at $b$ and $1/b$. Denote the Siegel disk centered at $0$, the Julia set, the filled-in Julia set, and the attracting basin of infinity for $f_b$ by $D(f_b)$, $J(f_b)$, $K(f_b)$, and $B_\infty(f_b)$ respectively. Then we have $1/b \in \partial D(f_b)$. Let $v_b := f_b(b)$ be the free critical value of $f_b$. The {\it cocritical point $co_b$ of $f_b$} is defined as the non-critical preimage of $v_b$ (i.e. $f_b(co_b) = v_b$ and $co_b \neq b$). Lastly, let $\phi^0_b: D(f_b)\to \D$ be the linearization map on $D(f_b)$, normalized by $\phi^0_b(1/b)=1$. 

\begin{prop}[{\cite[Cor. 13.6]{Za2}}]\label{prop.surgery}
For $b\in\overline{\mathcal{U}_\infty}$, there exist
\begin{itemize}
    \item a degree 5 Blaschke product $Bl$ of the following form 
\[Bl:z\mapsto e^{2\pi it}z^3\left(\frac{z-p}{1-\overline{p}z}\right)\left(\frac{z-q}{1-\overline{q}z}\right),\quad |p|,|q|>1\]
such that $1$ is a double critical point of $Bl$.
\item a quasi-regular polynomial $P$ with $P=Bl$ on $\mathbb{C}\setminus\overline{\D}$, $P = h^{-1} \circ R_\theta\circ h$ on $\overline{\D}$, where $h:\overline{\D}\to\overline{\D}$ is continuous, quasi-symmetric on $\partial\D$ and quasi-conformal on $\D$; $R_\theta$ is the rigid rotation of angle $\theta$.
\item a quasi-conformal homeomorphism $\varphi:\mathbb{C}\to\mathbb{C}$ such that $\varphi(1/b) = 1$, $\varphi$ is conformal on $B_\infty(f_b)$, $\varphi^{-1}\circ h^{-1}$ is conformal on $\D$, and $f_b = \varphi^{-1}\circ P\circ \varphi$.
\end{itemize}
\end{prop}

\begin{cor}\label{cor.quasi Siegel boundary}
For $b\in\overline{\mathcal{U}_\infty}$, the Siegel boundary $\partial D(f_b)$ of $f_b$ is a quasi-circle containing at least one critical point. Consequently, the linearization map $\phi^0_b$ extends to a homeomorphism from $\overline{D(f_b)}$ to $\overline{\bbD}$.
\end{cor}

\begin{prop}[{\cite[Remark on pp. 229]{Za2}}]\label{holo motion}
For $b\in\overline{\mathcal{U}_\infty}$, the closed Siegel disk $\overline{ {D(f_b)}}$ moves holomorphically. For a fixed starting point $b_0 \in\U_\infty$, the holomorphic motion is given by $(\phi^0_b)^{-1}\circ \phi^0_{b_0}$.
\end{prop}

\subsection{Dynamical Bubble rays}\label{sec:dym bubbleray} 

A connected component $B$ of $f_b^{-n}(D(f_b))$ for some $n \geq 0$ is called a {\it bubble}. An iterated preimage $x$ of the critical point $1/b$ is called a {\it joint}. The {\it generations} of $B$ and $x$ are defined the same way as for the quadratic Siegel polynomial $\q$.

We record the following obvious property about bubbles.

\begin{lem}\label{root}
Let $B$ be a bubble of generation $n \geq 1$. Suppose that $f_b^i(B) \not\ni b$ for $i \geq 0$. Then $\partial B$ contains a unique joint $x$ with the smallest generation, which is either $n$ or $n-1$.
\end{lem}

Let $B$ be a bubble as in \lemref{root}. If gen$(x) =n-1$ and $f_b^i(x) \neq b$ for $i \geq 0$, we refer to $x$ as the {\it root} of $B$, and denote $\rt(B) := x$. For $N \in \bbN \cup \{\infty\}$, a {\it bubble ray $R^\bB \sim \{B_i\}_{i=0}^N$ for $f_b$} is defined the same way as for $\q$. When the fact that $\bR$ is a bubble ray in the dynamical plane for $f_b$ needs to be emphasized, we denote $\bR = \bR(f_b)$. 

\begin{prop}[{\cite[Prop. 4.2]{Y}}]\label{prop.rational-bubble}
Rational bubble rays land. The landing points are (pre-)periodic.
\end{prop}

The {\it tree of bubbles $T(f_b) \subset K(f_b)$ for $f_b$} is defined to be the union of all bubble rays for $f_b$. The following result identifies a tree of bubbles for a cubic Siegel polynomial as a part of the filled Julia set for the quadratic Siegel polynomial.

\begin{thm}\label{cubic tree model}
There exists a canonical continuous injective map $\psi_b : T(f_b) \to T({\q})$, conformal on the interior of $T(f_b)$, such that
$$
f_b|_{T(f_b)}= \psi_b^{-1} \circ \q \circ \psi_b.
$$
Moreover, if $\psi_b$ is not surjective, then one of the following cases hold:
\begin{enumerate}[i)]
\item $b \in \partial T(f_b)$ and $b$ is a joint; or
\item $b$ is contained in a bubble $B$ whose boundary intersect $T(f_b)$ at a joint.
\end{enumerate}
In either case, the critical value $v_b$ and the co-critical point $co_b$ are contained in $T(f_b)$.
\end{thm}

\begin{proof}
The conjugacy $\psi_b$ is defined inductively as follows. For $n \geq 0$, let $\bfT_n(f_b) \subset T(f_b)$ and $\bfT_n(\q) \subset T(\q)$ be the unions of closures of bubbles of generation at most $n$ for $f_b$ and $\q$ respectively. On $\bfT_0(f_b) = \overline{D(f_b)}$, define
$$
\psi_b|_{\overline{D(f_b)}} := (\phi^0_{q_\theta})^{-1}\circ \phi^0_b.
$$

Suppose that $\psi_b$ is defined on $\bfT_n(f_b)$. Suppose $v_b$ is not a joint in $\partial \bfT_n(f_b)$, nor is contained in a bubble in $\bfT_n(f_b)$ of generation $n$. Let $B$ be a bubble of generation $n+1$ such that $f_b(B) \subset \bfT_n(f_b)$. Then one of three cases must hold:
\begin{enumerate}[i)]
\item $\overline{B} \cap \bfT_n(f_b) = \varnothing$;
\item $\overline{B} \cap \bfT_n(f_b) = \{x\}$, where $x$ is an iterated preimage of $b \in \partial \bfT_n(f_b)$; or
\item $\overline{B} \cap \bfT_n(f_b) = \rt(B)$.
\end{enumerate}
In the first two cases, $B$ is not contained in $\bfT_{n+1}(f_b)$. In the latter case, there exists a unique bubble $B' \subset \bfT_{n+1}(\q)$ such that
$$
\partial B' \cap \bfT_n(\q) = \{\psi_b(\rt(B))\}.
$$
Hence, $\psi_b$ extends to $\overline{B}$ via:
$$
\psi_b|_{\overline{B}} := (\q|_{\overline{B'}})^{-1}\circ \psi_b \circ f_b|_{\overline{B}}.
$$
\end{proof}

Let $\bR(f_b)$ be an infinite bubble ray for $f_b$ corresponding to the bubble ray
$$
\bR_t(\q) := \psi_b(\bR(f_b))
$$
for $\q$, where $t \in \bbR/\bbZ$ is the external angle of the unique external ray for $\q$ that co-lands with $\bR_t(\q)$. When convenient, we will denote $\bR(f_b) = \bR_t(f_b)$.

The main result in \cite{Y} is the following.

\begin{thm}\label{yang lc}
    Let $P$ be a polynomial of any degree with a connected Julia set $J(P)$. If $P$ has a bounded type Siegel disk $D(P)$, then $J(P)$ is locally connected at $\partial D(P)$.  
\end{thm}

We record the following immediate corollary of \thmref{yang lc}.

\begin{cor}\label{tree loc conn}
The Julia set $J(f_b)$ is locally connected at every point in $T(f_b)$.
\end{cor}

\subsection{Combinatorial Rigidity: Special case}\label{sec:combina.Zak}

\begin{prop}[{\cite[Cor. 5.2]{Za2}}]\label{prop.quasiconformal.rigid}
The boundary of the cubic Siegel connectedness locus $\mathcal{C}(\theta)$ is quasiconformally rigid. More precisely, if $f_{b_1}$ and $f_{b_2}$ are quasiconformally conjugate with $b_1,b_2\in\partial\mathcal{C}(\theta)$, then either $b_1 = \pm b_2$ or $b_1 = \pm1/b_2$.
\end{prop}

\begin{notn}
    For $i=1,2$, let $W_i,X_i$ be path connected topological spaces, $p_i:W_i\to X_i$ be covering maps. Let $f:X_1\to X_2$ be a continuous map. A continuous map $\Tilde{f}:W_1\to W_2$ is called a {\it$(p_1,p_2)$-lifting of $f$} if $p_2\circ\Tilde{f} = f\circ p_1$. 
\end{notn}

Let us recall from basic topological covering theory that for $z_1\in W_1$ and $z_2\in W_2$ with $f(p_1(z_1)) = p_2(z_2)$, $f$ admits a $(p_1,p_2)$-lifting $\Tilde{f}$ such that $\Tilde{f}(z_1) = z_2$, if and only if 
\begin{equation}\label{eq.lifting-condition}
(f\circ{p_1})_*\pi_1(W_1,z_1)\subset {p_2}_*\pi_1(W_2,z_2).
\end{equation}
Here $\pi_1$ denotes the fundamental group.

Let $b_1,b_2\in\overline{\mathcal{U}_\infty}\cap\partial\mathcal{C}(\theta)$. Suppose that ${co_{b_i}} \in T(f_{b_i})$ for $i \in \{1,2\}$. We say that $f_{b_1}$ and $f_{b_2}$ are {\it combinatorially equivalent}, if $\psi_{b_1}(co_{b_1}) = \psi_{b_2}(co_{b_2})$, where $\psi_{b_i}$ is given in \thmref{cubic tree model}. The following result states that  combinatorial rigidity holds in this special case:

\begin{thm}\label{thm.rigid.Zak}
Let $b_1,b_2\in\overline{\mathcal{U}_\infty}\cap\partial\mathcal{C}(\theta)$ satisfy the hypothesis above, then $f_{b_1}$ and $f_{b_2}$ are combinatorially equivalent if and only if $b_1=\pm b_2$.
\end{thm}
\begin{proof}
    We will use the "pull-back argument" presented in \cite[38.5]{Ly2}. Let ${Bl}_i,P_i,h_i,\varphi_i$ be as in 
 Proposition \ref{prop.surgery} associated to $f_{b_i}$. Notice that $\varphi_i\circ h_i^{-1} = \phi^{-1}_{b_i}$, $i=1,2$. Denote by $cp_i,v_{cp_i},{cop}_i$ respectively $\varphi^{-1}_i(b_i),\varphi^{-1}_i(v_{b_i}),\varphi^{-1}_i({c}_{b_i})$.
 Let $N\geq 0$ be the smallest integer such that $v_{cp_1}\in P_1^{-N}(\overline{\D})$. Since $\psi_{b_1}(co_{b_1}) = \psi_{b_2}(co_{b_2})$, by Theorem \ref{cubic tree model}, there exists a homeomorphism $\psi:\mathbb{C}\to\mathbb{C}$ such that $\psi\circ P_1 = P_2\circ \psi$ on $P^{-N-1}_1(\overline{\D})$. 
 
 On the other hand, since $h^{-1}_2\circ h_1|_{\partial\D}$ is quasi-symmetric, by Ahlfors-Beurling extension theorem, $h^{-1}_2\circ h_1$ admits a quasiconformal extension $H_0:{\C}\to{\C}$ such that $H_0(v_{cp_1}) = v_{cp_2}$. We modify quasiconformally $H_0$ on $\mathbb{C}\setminus\overline{\mathbb{D}}$ so that it is isotopic to $\psi$ rel $\{v_{cp_1}\}\cup\overline{\D}$. Suppose that $H_{n-1}$ is constructed such that $H_{n-1}$ is isotopic to $\psi$ rel $\{v_{cp_1}\}\cup P^{-(n-1)}_1({\overline{\D}})$. We construct inductively $H_n$ for $n\leq N$. One can check the lifting condition (\ref{eq.lifting-condition}):
    \[(H_{n-1}\circ{P_1})_*\pi_1(\C\setminus P_1^{-1}(\{P_1(1),v_{cp_1}\}),0)\subset {P_2}_*\pi_1(\C\setminus P_2^{-1}(\{P_2(1),v_{cp_2}\}),0).\]
To see this, firstly by $\psi\circ P_1 = P_2\circ \psi$ on $P^{-N-1}_1(\overline{\D})$, we have 
  $$(\psi\circ P_1)_*\pi_1(\C\setminus P_1^{-1}(\{P_1(1),v_{cp_1}\}),0) = (P_2)_*\pi_1(\C\setminus P_2^{-1}(\{P_2(1),v_{cp_2}\}),0).$$ 
  Since by construction, $H_{n-1}$ is isotopic to $\psi$ rel $\{v_{cp_1}\}\cup P^{-(n-1)}_1({\overline{\D}})$. Thus we can replace in the above equality $\psi$ by $H_{n-1}$.
    
   Thus $H_{n-1}$ admits a $(P_{1},P_{2})$-lifting $H_n$ such that $H_n(0) = 0$, $H_n \circ P_{1}= P_{2}\circ H_n$ on $P_1^{-n}(\overline{\D})$. Modify quasiconformally $H_n$ on $\mathbb{C}\setminus P_1^{-n}(\overline{\D})$ such that $H_n(v_{cp_1})=v_{cp_2}$ and $H_n$ is isotopic to $\psi$ rel $\{v_{cp_1}\}\cup P_1^{-n}(\overline{\D})$. We finish the construction of $H_n$.
    
    Let $\Tilde{H}_N := \varphi_2\circ H_N\circ\varphi^{-1}_1$. By construction, $\Tilde{H}_N \circ f_{b_1}= f_{b_2}\circ \Tilde{H}_N$ on $f_{b_1}^{-N}(\overline{D(f_{b_1})})$.  By hypothesis $\psi_{b_1}(co_{b_1}) = \psi_{b_2}(co_{b_2})$, $[v_{{b}_1}]=[v_{{b}_2}]$, thus we necessarily have $\Tilde{H}_N(v_{b_1})=v_{b_2}$. Let $\phi^\infty_{b_i}:B_\infty(f_{b_i})\to\C\setminus\overline{\D}$ be the Böttcher coordinate of $f_{b_i}$. Fix some $r>1$. Modify quasiconformally $\Tilde{H}_N$ on $(\phi^\infty_{b_1})^{-1}(\{z;\,|z|> r\})$ such that $\Tilde{H}_N = (\phi^\infty_{b_2})^{-1}\circ\phi^\infty_{b_1}$ on $(\phi^\infty_{b_1})^{-1}(\{z;\,|z|> r^3\})$. Again $\Tilde{H}_N$ admits a $(f_{b_1},f_{b_2})$-lifting $\Tilde{H}_{N+1}$ such that $\Tilde{H}_{N+1}(0) = 0$ (the lifting condition is satisfied since $\psi_{b_1}(co_{b_1}) = \psi_{b_2}(co_{b_2})$). Notice that since $f_{b_i}$ are conformal, $\Tilde{H}_{N+1}$ has the same dilatation as $\Tilde{H}_N$. Moreover $\Tilde{H}_{N+1}(v_{b_1}) = v_{b_2}$. Thus we can lift it again to $\Tilde{H}_{N+2}$. Repeat this procedure we get a sequence of quasiconformal map $\{\Tilde{H}_n\}_{n\geq N}$ with bounded dilatation. By compactness of quasiconformal maps with bounded dilatation, there is a subsequence converging to some quasiconformal map $\Tilde{H}$. By construction, $f_{b_1}\circ\Tilde{H}=\Tilde{H}\circ f_{b_2}$ on $\C\setminus J(f_{b_1})$, hence on $\C$ by continuity. By Proposition \ref{prop.quasiconformal.rigid}, $b_1=\pm b_2$.
\end{proof}

\section{Dynamical External Rays}\label{sec:ext rays}

Consider a cubic Siegel polynomial $f_b \in \cP^{cm}(\theta)$ for some $b \in \overline{\cU_\infty}$. On a suitable neighborhood $V_b \subset \hat\bbC$ of $\infty$, the dynamics of $f_b$ uniformizes to that of $w \mapsto w^3$ via the B\"ottcher map $\phi^\infty_b : V_b \to \bbC \setminus \overline{\bbD}$ (normalized by $(\phi^\infty_b)'(\infty) > 0$).

For $t > 0$, denote $\bbD_t := \{|z| < t\}.$ Recall that if $b \not\in B_\infty(f_b)$, then $\phi^\infty_b$ extends to a conformal map between $B_\infty(f_b)$ and $\mathbb{C} \setminus \overline{\mathbb{D}}$. Otherwise, there exists a maximal radius $\rho_b >1$ such that $(\phi^\infty_b)^{-1}$ extends to a conformal map on $\mathbb{C} \setminus \overline{\bbD_{\rho_b}}$. The set $(\phi^\infty_b)^{-1}(\partial \bbD_{\rho_b})$ is a figure-eight curve which contains $b$ at the self-intersection point. While $\phi^\infty_b$ is not well-defined at $b$, it does extend conformally to a neighborhood of the cocritical point $co_b$.

Let $R^\infty_t$ and $E_l^\infty$ be the external ray with external angle $t \in \bbR/\bbZ$ and the equipotential at level $l > 1$ respectively. Then we have
$$
f_b(R^\infty_t) = R^\infty_{3t}
\matsp{and}
f_b(E^\infty_l) = E^\infty_{l^3}.
$$
When the fact that $R^\infty_t$ is an external ray in the dynamical plane for $f_b$ needs to be emphasized, we will denote $R^\infty_t = R^\infty_t(f_b)$. For example, there are exactly two fixed external rays $R^\infty_0(f_b)$ and $R^\infty_{1/2}(f_b)$ for $f_b$. A point $x \in J(f_b)$ is said to be {\it $n$-accessible} if exactly $n$ external rays co-land at $x$.

We record the following consequence of \corref{tree loc conn}.

\begin{cor}\label{access to tree}
Every point $x \in T(f_b)$ is uni-accessible, unless $x$ is a root point, in which case it is bi-accessible.
\end{cor}

Suppose that there are two distinct external rays $R^\infty_s$ and $R^\infty_t$ that co-land at a point $x \in J(f_b)$. The connected component of $\bbC \setminus (R^\infty_s \cup R^\infty_t \cup \{x\})$ not containing $D(f_b)$ is called a {\it dynamical wake}, and is denoted $W_{s,t}(f_b)$. The point $x$ is referred to as the {\it base point of $W_{s,t}(f_b)$}.

Let $R^{\bfB}\sim \{B_i\}_{i=0}^\infty$ be an infinite bubble ray for $f_b$. For $i \geq 1$, let $W_{s_i, t_i}(f_b)$ be the dynamical wake whose base point is given by $x_i := \rt(B_i)$. Note that $\{W_{s_i, t_i}(f_b)\}_{i=1}^\infty$ is nested. The set
$$
\Imp_W(R^{\bfB}) := \bigcap_{i=1}^\infty W_{s_i, t_i}(f_b) \cap K(f)
$$
is called the {\it wake impression of $R^{\bfB}$}. As $i \to \infty$, the external angles $s_i$ and $t_i$ converge to $s_*$ and $t_*$ respectively. We refer to
$$
\Arc_W(R^{\bfB}) := [s_*, t_*]\subset \bbR/\bbZ
$$
as the {\it wake arc of $R^\bfB$}. We say that $\Arc_W(R^{\bfB})$ is {\it trivial} if it is a singleton.

\begin{prop}\label{prop.wake-impression}
Suppose $co_b \not\in T(f_b)$. Then we have
$$
K(f_b) = T(f_b) \sqcup \{\emph{wake impressions of infinite bubble rays}\}.
$$
\end{prop}

\begin{proof}
Assume towards a contradiction that $x \in J(f_b)$ is not contained in any dynamical wake based at a root point in $T(f_b)$. Denote
$$
c_{-n} := (f_b|_{\partial D(f_b)})^{-1}(1/b)
\matsp{for}
n \geq 0.
$$
Choose two sequences $\{n_i\}_{i=1}^\infty$ and $\{m_i\}_{i=1}^\infty$ such that for $i \in \bbN$, we have the following properties.
\begin{itemize}
\item For $k,l \geq 0$, denote the subarc of $\partial D(f_b)$ with endpoints $c_{-k}$ and $c_{-l}$ by $(-k, -l)_c$. Then $(-n_i, -m_i)_c \Supset (-n_{i+1}, -m_{i+1})_c$.
\item For $k \geq 0$, denote the dynamical wake based at the root point $c_{-k}$ by $W_{-k}$. Let $X_i$ be the connected component of $\bbC \setminus (W_{-n_i} \cup W_{-n_i} \cup \overline{D(f_b)})$ containing $x$. Then $X_i \supset X_{i+1}$.
\end{itemize}

For $i \in \bbN$, let $s_i, t_i \in \bbR/\bbZ$ be the external angles such that
$$
R^\infty_{s_i} \subset \partial X_i \cap \partial W_{-n_i}
\matsp{and}
R^\infty_{t_i} \subset \partial X_i \cap \partial W_{-m_i}.
$$
As $i \to \infty$, the points $c_{-n_i}$ and $c_{-m_i}$ converge to some common point $z_* \in \partial D(f_b)$; and the external angles $s_i$ and $t_i$ converge to some limit angles $s_*$ and $t_*$ respectively. By \thmref{yang lc}, it follows that the external rays $R^\infty_{s_*}$ and $R^\infty_{t_*}$ co-land at $z_*$. By the assumption, $z_*$ cannot be a root, and hence, by \corref{access to tree}, $z_*$ must be uni-accessible. However, removing $z_*$ disconnects $x$ from $\partial D(f_b)$. This is a contradiction.
\end{proof}

\begin{prop}\label{prop.irrational wake trivial}
Consider an irrational infinite bubble ray $\bR$. Then there exists $N \geq 0$ such that for $n \geq N$, the wake arc $\Arc_W(f_b^n(R^{\bfB}))$ is trivial.
\end{prop}

\begin{proof}
Suppose that $\Arc_W(f_b^n(R^{\bfB})) =: [s_n, t_n]$ is non-trivial for all $n \in \bbN$. Observe that
$$
I_n := R^\infty_{s_n} \cup R^\infty_{t_n} \cup \Imp_W(f_b^n(R^{\bfB}))
$$
disconnects $\bbC$. Moreover, if $C_n$ is the connected component of $\bbC \setminus I_n$ that does not contain $D(f_b)$, then $C_n$ is simply-connected. If $C_n \not\ni b$, then $f_b$ maps $C_n$ conformally onto $C_{n+1}$. Thus, $[s_{n+1}, t_{n+1}] = [3s_n, \rz{3}t_n]$. This clearly cannot hold for all $n \in \bbN$.
\end{proof}

Recall that a holomorphic degree 2 branched cover $g : U’ \to U \Supset U'$ defined on a Jordan domain $U’ \subset \bbC$ is called a {\it quadratic-like map}. The filled Julia set of $g$ is defined as
$$
K(g) := \bigcap_{n=1}^\infty g^{-n}(U).
$$
There exists a unique fixed point $\beta(g) \in K(g)$ such that $g’(\beta(g)) \geq 1$.

\begin{prop}\label{prop.separate}
Suppose there exists a $p$-periodic bubble ray $\bR$ for $f_b$ landing at $x \in J(f_b)$ with a non-trivial wake arc $\Arc_W(R^{\bfB}) = [s,t]$. Then the following statements hold.
\begin{enumerate}[i)]
\item The external rays $R^\infty_s$ and $R^\infty_t$ are $p$-periodic, and are the only  external rays that land at $x$.
\item There exists $0 \leq n < p$ such that the critical point $b$ is contained in $W_{3^ns,3^nt}(f_b)$. Moreover, there exists a Jordan domain $U’ \supset \{f^n(x), b\}$ such that $g:= f_b^p|_{U’}$ is a quadratic-like map, and $\beta(g) = f^n(x)$.
\end{enumerate}
\end{prop}

\begin{proof}
Since $\bR$ maps to itself under $f_b^p$, we must have $(f_b^p)'(x) \geq 1$. It follows that $R^\infty_s$ and $R^\infty_t$ must each map to itself under $f_b^p$ also.

Consider the orbit of the dynamical wake $f_b^n(W_{s, t}(f_b))$ for $0 \leq n < p$. Arguing as in \propref{prop.irrational wake trivial}, we see that there exists the smallest value of $n$ such that $W_{3^ns, 3^nt}(f_b) = f_b^n(W_{s, t}(f_b))$ contains the critical point $b$. By replacing $\bR$ with $f_b^n(\bR)$ if necessary, we may assume that $n = 0$.

Denote
$$
I_0 := R^\infty_s \cup R^\infty_t \cup \{x\}.
$$
Using a nearly identical argument as in the proof of Lemma 8.1 in \cite{Mi}, we see that by taking the connected component $\tilde U'$ of
$$
\bbC \setminus (f_b^{-p}(I_0) \cup E^\infty_l)
$$
containing $b$ (where $l >1$ is taken arbitrarily), then thickening it slightly, we obtain a Jordan domain $U'$ on which $g := f_b^p|_{U'}: U' \to U$ is a quadratic-like map. The fact that $\beta(g) = x$, and that $R^\infty_s$ and $R^\infty_t$ are the only external rays for $f_b$ that land at $x$ now follow.
\end{proof}

A cubic Siegel polynomial $f_b$ is said to be {\it separable} if it has a periodic bubble ray with a non-trivial wake arc.

\begin{prop}\label{nonsep}
Suppose that $f_b$ is non-separable and has a connected Julia set. If $b \not\in T(f_b)$, then either $b$ is the landing point of some strictly pre-periodic bubble ray (and hence, $f_b$ is Misiurewicz), or $b$ is contained in the wake impression of an irrational bubble ray.
\end{prop}

\begin{proof}
By \propref{prop.wake-impression}, there exists an infinite bubble ray $\bR$ whose wake impression contains the critical point $b$. If $\bR$ is periodic, then its wake arc cannot be trivial, since otherwise, $b$ would be a periodic co-landing point of $\bR$ and an external ray $R^\infty_t$. Hence, $f_b$ is separable by \propref{prop.separate}. If $\bR$ is strictly pre-periodic, then the wake arc of $f_b^i(\bR)$ for $i \in \bbN$ must be trivial, since its wake impression is disjoint from all critical points. In this case, $\bR$ must land at $b$.
\end{proof}

\section{Cubic Siegel Parameter Space}\label{sec:para bubble rays}

Consider the cubic Siegel familiy $\cP^{cm}(\theta)$ given by (\ref{eq.family.fb}). A {\it parabubble} $\cB$ is a connected component of the set
$$
\{b\in\mathcal{U}_\infty \; | \; f_b^n(b)\in D(f_b) \; \text{ for some } \; n \in \bbN\}.
$$

For $t >0$ and $z_0 \in \bbC$, denote $\D_t(z) := \{|z-z_0| < t\}$.

\begin{lem}[{{\cite[Lem. 3.2]{Sh}}, \cite[Lem. 4.1]{WaYan}}]\label{lem.transversality}
    Let $h:\D_{s_0}(0)\times X\to\C$ be a holomorphic motion of the set $X\subset\C$. Suppose that $\mathrm{a}:\D_{s_0}\to\C$ is an analytic map satisfying $a(0) = z_0\in X$ and $\mathrm{a}(\cdot)\not\equiv h(\cdot,z_0)$. Then there exists $r_0>0$ and $0<s <s_0$ such that for all $0<r\leq r_0$, the set
    $$Y_r = \{\lambda\in\D_s(0) \; | \;\mathrm{a}(\lambda)\in h(\lambda,X\cap\D_r(z_0))\}$$
    is mapped onto $X\cap\D_r(z_0)$ by the restriction of a quasi-regular map with the form $\mathrm{a}\circ (\varphi_r)^{-1}$, where $\varphi_r:\C\to\C$ is quasi-conformal.
\end{lem}

\begin{lem}\label{arc partial}
Let $b_0\in\mathcal{U}_\infty$. If $b_0$ belongs to the boundary of some parabubble $\cB$, then the critical point $b_0$ is contained in the boundary of a bubble $B_{b_0}$. In this case, there exists a neighborhood $\V_0$ of $b_0$ such that $\V_0\cap\partial\cB$ is a quasi-arc.
\end{lem}

\begin{proof}
The first assertion follows immediately from the fact that the Siegel disk moves continuously with
respect to the parameter $b$ \cite[Thm. 13.9]{Za2}. Thus we can set $k\geq 1$ to be the first moment such that $f_{b_0}^k(b_0)\in\partial D(f_{b_0})$. We can apply Lemma \ref{lem.transversality} with $\mathrm{a}(b) = f^k_b(b)$, $X = \partial D(f_{b_0})$, $h$ being the holomorphic motion of $\partial D(f_b)$ with base point $b_0$ (Proposition \ref{holo motion}). Let $z_0 = \mathrm{a}(b_0)$, $\tilde{X}$ be the connected component of $X\cap D(z_0,r_0)$ containing $z_0$, $\Tilde{Y}^{r_0}$ the connected component of $Y^{r_0}$ containing $b_0$. Since $\tilde{X}$ is a quasi-arc, $\Tilde{Y}^{r_0}$ is the union of $m$ quasi-arcs intersecting transversally at $b_0$, where $m$ is the degree of $\mathrm{a}$ at $b_0$. Take $\mathcal{V}_0$ small enough such that $\mathcal{V}_0\cap\partial\mathcal{B}\subset \Tilde{Y}^{r_0}$. Notice that $\partial\mathcal{B}$ has no isolated point (since $\partial\pB$ is simply connected), $\mathcal{V}_0\cap\partial\mathcal{B}$ is the union of at least 1 components of $\Tilde{Y}^{r_0}\setminus b_0$, in particular, $\partial\mathcal{B}$ is locally connected at $b_0$. By \cite[Thm. 3.1]{Za2}, the number of these components is exactly 2, hence $\mathcal{V}_0\cap\partial\mathcal{B}$ is a quasi-arc. Thus if $\partial\pB\cap\Zak = \emptyset$, $\partial\pB$ is a quasi-circle.
\end{proof}

\begin{prop}\label{unif cap}
Let $\cB \subset \cU_\infty$ be a parabubble so that for $b \in \cB$, we have $co_b \in T(f_b)$. Define $\Psi : \cB \to \Psi(\cB) \subset T(\q)$ by $\Psi(b) := \psi_b(co_b)$, where $co_b$ is the co-critical point and $\psi_b$ is given in \thmref{cubic tree model}. Then $\Psi$ is a conformal map between $\cB$ and the bubble $\Psi(\cB)$ for $\q$. Moreover, $\Psi$ extends to a homeomorphism between $\overline{\cB}$ and $\overline{\Psi(\cB)}$. 
\end{prop}

\begin{proof}
For $b \in \cB$, the co-critical point $co_b$ is contained in some bubble $B_b \subset T(f_b)$. Moreover, there exists a bubble $B$ for $\q$ independent of $b$ such that $\psi_b(B_b) = B$. It follows from the proof of Lemma 5.4 in \cite{Za2} that $\Psi|_{\cB}$ is a proper holomorphic map from $\cB$ onto $B$. Additionally, \thmref{thm.rigid.Zak} implies that $\Psi|_{\cB}$ must be injective. It remains to show that $\Psi$ extends continuously to $\partial \cB$. If $\partial \cB \cap \cZ_\theta = \varnothing$, then the result follows immediately from \lemref{arc partial}.

Suppose that $\partial \cB \cap \cZ_\theta \neq \varnothing$. Denote the internal rays of angle $t \in \bbR/\bbZ$ in $B \subset T(\q)$ and $B_b \subset T(f_b)$ by $R^B_t \subset B$ and $R^{B_b}_t \subset B_b$ respectively. Define the parameter internal ray of angle $t$ as $\cR^{\cB}_t := \Psi^{-1}(R^B_t)$. If $b\in \cR^{\cB}_t$, then in the dynamical plane for $f_b$, the co-critical point $co_b$ is contained $R^{B_b}_t$.

Let $t$ be an internal angle such that the intersection between the accumulation set $\omega(\cR^{\cB}_t)$ and $\cZ_\theta$ is non-empty, and hence, contains some parameter $b_t$. By the properness of $\Psi$, it follows that in the dynamical plane for $f_{b_t}$, the co-critical point $co_{b_t}$ is contained in the boundary of the bubble $B_{b_t} := \psi_{b_t}^{-1}(B)$. Let $D'(f_{b_t})$ be the unique bubble of generation $1$ rooted at $1/b_t$ for $f_{b_t}$. Evidently, $b_t \in \partial D(f_{b_t})$ if and only if $co_{b_t} \in \partial D'(f_{b_t})$. Moreover, the latter is true if and only if $t=0$ and
$$
f_{b_t}^{n-1}(co_{b_t}) = f_{b_t}^{n-1}(b_t) = 1/b_t,
$$
where $n = \gen(B_{b_t}) = \gen(B)$. It follows that $\omega(\cR^{\cB}_0)$ is a discrete set, and hence, a singleton $\{b_0\}$. Moreover, if $t \neq 0$, then $b_t \not\in \cZ_\theta$, and $\cR^{\cB}_t$ lands at $b_t$ by \lemref{arc partial}. We conclude that $\cZ_\theta \cap \partial \cB = \{b_0\}$, and the result follows.
\end{proof}

Consider the filled Julia set $K(\q)$ of the quadratic Siegel polynomial $\q$. Recall that the critical point of $\q$ is at $1$. For $m \in \bbZ$, denote
$$
c_m(\q) := \q|_{\partial D(\q)}^m(1).
$$
For $n \geq 0$, define the {\it limb $L_{-n}(\q)$ rooted at $c_{-n}(\q)$} as the union of $\{c_{-n}(\q)\}$ and the connected component of $T(\q) \setminus c_{-n}(\q)$ that does not contain $D(\q)$.

\begin{prop}\label{limb ext}
Let $b_0 \in \cU_\infty$ be a parameter such that $co_{b_0} \in T(f_{b_0})$. Let $n \geq 0$ be the integer such that $\psi_{b_0}(co_{b_0}) \in L_{-n}(\q)$. Then there exists a homeomorphism
$$
\Psi^{-1} : L_{-n}(\q) \to \cL_{-n} \subset \cU_\infty \sqcup \cZ_\theta
$$
such that for any bubble $B \subset L_{-n}(\q)$ for $\q$, the set $\cB := \Psi^{-1}(B)$ is a parabubble, and $\Psi|_{\overline{\cB}}$ coincides with the map given in \propref{unif cap}. Additionally, we have
$$
\{\zeta_{-n} := \Psi^{-1}(c_{-n}(\q))\} = \cL_{-n} \cap \cZ_\theta,
$$
and for $b_* := \zeta_{-n}$, we have $f_{b_*}^n(b_*) = 1/b_*$.
\end{prop}

\begin{proof}
By \lemref{arc partial}, the parameter $b_0$ is contained in the closure of a parabubble $\cB_0 \subset \cU_\infty$ if and only if in the dynamical plane for $f_{b_0}$, the co-critical point $co_{b_0}$ is contained in the closure of some bubble $B_{b_0} \subset T(f_{b_0}) \setminus D(f_{b_0})$.

Consider the bubble
$$
B_0 :=\Psi(\cB_0) = \psi_{b_0}(B_{b_0}) \subset T(\q)
$$
for $\q$. Let $x_1 \in \partial B_0$ be a root point for $\q$. Then there exists a unique bubble $B_1 \subset T(\q)$ such that $\partial B_0 \cap \partial B_1 = \{x_1\}$. Denote $b_1 := \Psi^{-1}(x_1)$.

We have $\partial \Psi_{b_1}^{-1}(B_0) \cap \partial \Psi_{b_1}^{-1}(B_1) = \{b_1\}$. If $B_1 \neq D(\q)$, then by the above observation, there exists a parabubble $\cB_1 \subset \cU_\infty$ such that $\partial \cB_0 \cap \partial \cB_1 = \{b_1\}$. Moreover, $\cB_1$ must be unique by \thmref{thm.rigid.Zak}. The result follows.
\end{proof}

\subsection{Parameter external rays}

The {\it cubic Siegel escape locus} is the set $\cA(\theta) := \bbC^* \setminus \cC(\theta)$. We denote the unbounded component of $\cA(\theta)$ by $\cA_\infty(\theta)$.

\begin{thm}[\cite{Za1} Theorem 6.1]\label{uniformization of escape}
Let $\Phi^\infty : \cA_\infty(\theta) \to \mathbb{C} \setminus \overline{\mathbb{D}}$ be a map defined by
$$
\Phi^\infty(\b) := \phi^\infty_\b(co_\b),
$$
where $co_\b$ is the co-critical point of $f_b$. Then $\Phi^\infty$ is a biholomorphism.
\end{thm}

For $t \in \bbR/\bbZ$ and $l \geq 1$, the curves
$$
\mathcal{R}^{\infty}_t := \{\text{arg}(\Phi^\infty) = t\}
\matsp{and}
\cE^\infty_l := \{|\Phi^\infty)| = l\}
$$
are called the {\it parameter external ray with external angle $t$} and the {\it parameter equipotential at level $l$} respectively.

Suppose that there are two distinct parameter external rays $\cR^\infty_s$ and $\cR^\infty_t$ that co-land at a parameter $x \in \cC(\theta)$. The connected component of $\bbC \setminus (\cR^\infty_s \cup \cR^\infty_t \cup \{x\})$ not containing $\cZ_\theta$ is called a {\it parameter wake}, and is denoted $\cW_{s,t}$. The point $x$ is referred to as the {\it base point of $\cW_{s,t}$}.

For $(s,t) \subset \bbR/\bbZ$, denote
$$
\cV_{s,t} := \bigcup_{u \in (s,t)} \cR^\infty_u \subset \cA_\infty(\theta).
$$

We say that $b$ and $f_b$ are {\it parabolic} if $f_b$ has a parabolic periodic orbit. Similarly, we say that $ b$ and $f_b$ are {\it Misiurewicz} if $b$ is strictly pre-periodic under $f_b$.

\begin{lem}\label{fund wakes}
The parameter external rays $\cR^\infty_{2/3}$ and $\cR^\infty_{5/6}$ co-land at some parabolic parameter $b_* \in \cC_\infty(\theta)$, creating a parameter wake $\cW_{\frac{2}{3}, \frac{5}{6}}$. Similarly, $\cR^\infty_{1/6}$ and $\cR^\infty_{1/3}$ co-land at $-b_*$, creating a parameter wake $\cW_{\frac{1}{6}, \frac{1}{3}}$. Furthermore, we have the following landing patterns in the dynamical plane.
\begin{enumerate}[i)]
\item For $b \in \cW_{\frac{1}{6}, \frac{1}{3}} \cup \cW_{\frac{2}{3}, \frac{5}{6}}$, the bubble ray $\bR_0(f_b)$ and the external rays $R^\infty_0(f_b)$ and $R^\infty_{1/2}(f_b)$ co-land at a bi-accessible repelling fixed point.
\item For $b \in \cV_{\frac{5}{6}, \frac{1}{6}}$, the bubble ray $\bR_0(f_b)$ and the external ray $R^\infty_0(f_b)$ co-land at a uni-accessible repelling fixed point.
\item For $b \in \cV_{\frac{1}{3}, \frac{2}{3}}$, the bubble ray $\bR_0(f_b)$ and the external ray $R^\infty_{1/2}(f_b)$ co-land at a uni-accessible repelling fixed point.
\end{enumerate}
\end{lem}

\begin{proof}
The most part of the Lemma is due to Buff-Henriksen \cite{BuHe}. The colanding of the parameter external rays is justified by \cite[Prop. 8]{BuHe}. When $b$ belongs to $\mathcal{W}_{\frac{1}{6},\frac{1}{3}}$ or $\mathcal{W}_{\frac{2}{3},\frac{5}{6}}$,  the colanding of the corresponding two dynamical external rays is justified by \cite[Prop. 9]{BuHe}. Notice that for $b\in(\mathcal{W}_{\frac{1}{6},\frac{1}{3}}\cup\mathcal{W}_{\frac{2}{3},\frac{5}{6}})\cap\mathcal{C}(\theta)$, $R^\bB_{0}(f_b)$ is well-defined and lands at a repelling fixed point. This fixed point is landed at least by one rational external ray which is fixed under the action of $f_b$, since $R^\bB_{0}(f_b)$ is fixed by $f_b$. Thus the only possible external rays are $R^\infty_0(f_b)$ and $R^\infty_{1/2}(f_b)$. This proves \romannumeral1).

Next we prove \romannumeral2) and it will be the same for \romannumeral3). Notice that for $b_0\in\mathcal{R}^\infty_{1/6}$, the dynamical ray $R^\infty_{1/2}(f_{b_0})$ crashes on the critical point $b_0$, while $R^\infty_{0}(f_{b_0})$ still colands with $R^\bB_{0}(f_{b_0})$ at a repelling fixed point. By the stability of the landing property at a repelling fixed point, we conclude that for $b\in\mathcal{V}_{\frac{5}{6},\frac{1}{6}}$, $R^\infty_{0}(f_b)$ colands with $R^\bB_{0}(f_b)$ at the repelling fixed point $x_b$. It remains to verify that for $b\in\mathcal{V}_{\frac{5}{6},\frac{1}{6}}$, $R^\infty_{1/2}(f_b)$ does not land at $x_b$. Indeed, otherwise $R^\infty_{1/2}(f_b)$ and $R^\infty_{0}(f_b)$ would cut a wake $W$ which does not contain any critical point of $f_b$. This contradicts Proposition \ref{prop.separate} \romannumeral2).
\end{proof}

\begin{thm}\label{paratree}
The map $\Psi$ given in \propref{limb ext} extends to a two-to-one covering map from a set $\cT_\infty(\theta) \subset \cC_\infty(\theta)$ to $T(\q) \setminus D(\q)$. In fact, we have
$$
\cT_\infty(\theta)/[b \sim -b] \stackrel{\Psi}{\simeq} T(\q) \setminus D(\q).
$$ 
\end{thm}

\begin{proof}
We first choose a root point from each limb $L_{-n}(\q)$ for $n \geq 0$ that is not contained in $\partial D(\q)$. Let $B_{-n}(\q) \subset L_{-n}(\q)$ be the unique bubble whose boundary intersect $\partial D(\q)$ at an $n$th iterated preimage $c_{-n}(\q)$ of the critical point $1$. Let $\hc_{-n}(\q)$ be the unique root point in $\partial B_{-n}(\q)$ of generation $n+1$. In the dynamical plane for $f_b$ with $b\in \cU_\infty$, the corresponding objects for $f_b$ are defined as the pullbacks under $\psi_b$ (when well-defined).

Consider a circle of parameters near infinity:
$$
b_t := (\Phi^\infty)^{-1}(Re^{2\pi it})
\matsp{for}
t \in \bbR/\bbZ,
$$
where $R > 1$ is fixed. In the dynamical plane for $f_{b_t}$, denote the external rays that land at a root point $x \in T(f_{b_t})$ by $R^\infty_{s^l_t(x)}(f_{b_t})$ and $R^\infty_{s^r_t(x)}(f_{b_t})$ (the former approaches $x$ from the left and the latter from the right).

Note that $b_t \in \cS_{\fun}$ if and only if $t \in (-1/6, 1/6)$. Partition $\partial D(\q)$ into the disjoint union of two open subarcs $A_+(\q)$ and $A_-(\q)$ whose common endpoints are $c_0(\q)$ and $c_{-1}(\q)$, such that $c_1(\q) \in A_+(\q)$. We claim that for all $n \geq 0$, there exists a unique value of $t\in (-1/6, 1/6)$ such that one of the following possibilities hold.
\begin{enumerate}[i)]
\item If $n =0$ or $c_{-n}(\q) \in A_-(\q)$, then $t \in (0, 1/6)$, and the co-critical point $co_{b_t}$ is contained in the external ray landing at the root point $\hc_{-n}(f_{b_t})$ from the left:
$$
co_{b_t} \in R^\infty_{s^l_t(\hc_{-n}(f_{b_t}))}(f_{b_t}).
$$
\item If $n=1$ or $c_{-n}(\q) \in A_+(\q)$, then $t \in (-1/6, 0)$, and the co-critical point $co_{b_t}$ is contained in the external ray landing at the root point $\hc_{-n}(f_{b_t})$ from the right:
$$
co_{b_t} \in R^\infty_{s^r_t(\hc_{-n}(f_{b_t}))}(f_{b_t}).
$$
\end{enumerate}

First, consider the case $t \in (0, 1/6)$. In the dynamical plane for $f_{b_t}$, the external rays that co-land at $\bR_0(f_{b_t})$ and $\bR_{1/2}(f_{b_t})$ are $R^\infty_0(f_{b_t})$ and $R^\infty_{1/6}(f_{b_t})$ respectively (see \lemref{fund wakes}). The union of these two bubble rays and external rays separate the dynamical plane into two components: one to the left of $\bR_0(f_{b_t})$ and another to the right. Since the co-critical point $co_{b_t}$ must be contained in the former, the critical point $b_t$ must be contained in the latter.

To start, consider the dynamical plane for $f_{b_\epsilon}$ for some $0 < \epsilon \ll 1$. The points $co_{b_\epsilon}$ and $b_\epsilon$ are contained in the dynamical wake based at $\hc_0(f_{b_\epsilon})$ and $\hc_{-1}(f_{b_\epsilon})$ respectively. Now, as $t$ increases to $1/6$ from $\epsilon$, the critical point $b_t$ does not hit the external rays landing at $c_0(f_{b_t})$ or $\hc_0(f_{b_t})$ from the left. Hence, the landings of these rays persist, and it follows that
$$
s^l_t(c_0(f_{b_t})) \equiv s^l_\epsilon(c_0(f_{b_\epsilon})) =:  \tau_0
\matsp{and}
s^l_t(\hc_0(f_{b_t})) \equiv s^l_\epsilon(\hc_0(f_{b_\epsilon})) =:  t_0
$$
for all $\epsilon \leq t < 1/6$. Thus, $t_0$ is the unique value of $t\in (0, 1/6)$ such that claim i) holds for $n=0$.

Next consider $t \in (t_0, \tau_0)$. Then $co_{b_t}$ is contained in the dynamical wake based at $c_0(f_{b_t})$, but disjoint from the dynamical wake based at $\hc_0(f_{b_t})$. Thus, it does not hit an external ray landing at $\hc_{-n}(f_{b_t})$ from the left for any $n \geq 0$.

Lastly, consider $t \in (\tau_0, 1/6)$. Then in the dynamical plane for $f_{b_t}$, the critical point $b_t$ is confined to the dynamical wake based at $c_0(f_{b_t})$. Thus, $b_t$ does not interfere with the external ray landing at $\hc_{-n}(f_{b_t})$ from the left for any $n \geq 0$. Therefore, we have
$$
s^l_t(\hc_{-n}(f_{b_t})) \equiv s^l_{\tau_1+\epsilon}(\hc_{-n}(f_{b_{\tau_1+\epsilon}})) =: t_{-n}
\matsp{for}
\tau_1 < t < 1/6.
$$
Thus, $t_{-n}$ is the unique value of $t \in (0, 1/6)$ such that claim $i)$ holds for $n > 1$.

Claim ii) follows from an analogous argument.

Now, fix $n \geq 0$. For all $b \in \cR^\infty_{t_{-n}}$, the landing of the external ray $R^\infty_{t_{-n}}(f_{b_t})$ at $\hc_{-n}(f_{b_t})$ persists. Hence, it follows that $\cR^\infty_{t_{-n}}$ lands at a parameter $\hat b$, for which we have $co_{\hat b} = \hc_{-n}(f_{\hat b})$. The result now follows from \propref{limb ext}.
\end{proof}

The set $\cT_\infty(\theta)$ given in \thmref{paratree} is called the {\it (exterior) parameter tree}.  Let $\bR_t(\q)$ be a bubble ray for $\q$. The set
$$
\hat\cR_t^{\bfB} := \Psi^{-1}(\bR_t(\q) \setminus D(\q))
$$
is called a {\it parameter bubble dual-ray}. Let $\{\cB_i\}_{i=1}^\infty$ be a sequence of parabubbles such that
$$
\hat\cR^{\bfB}_t =\cZ_\theta \cup \bigcup_{i=1}^\infty (\cB_i \cup -\cB_i),
$$
and $\partial \cB_i \cap \partial \cB_{i+1} \neq \varnothing$ for $i \in \bbN$. We say that $\hat\cR^{\bfB}_t$ {\it lands at $x \in \cC_\infty(\theta)$} if $\cB_i$ or $-\cB_i$ converges to $x$ in the Hausdorff topology as $i \to \infty$. We refer to $\hat\cR^{\bfB}_{1/2}$ as the {\it parameter bubble spine}.

\begin{lem}\label{lem.fund region}
The parameter bubble spine $\hat\cR^{\bfB}_{1/2}$ lands at the base points $\{b_*, -b_*\}$ of the parameter wakes $\cW_{\frac{1}{6}, \frac{1}{3}}$ and $\cW_{\frac{2}{3}, \frac{5}{6}}$.
\end{lem}

\begin{proof}
Let $b_0$ be a parameter in the accumulation set $\omega(\hat\cR^{\bfB}_{1/2})$ of $\hat\cR^{\bfB}_{1/2}$. If $\bR_0(f_{b_0})$ lands at a repelling fixed point, then $\bR_0(f_b)$ would persistently land for $b$ in a small neighborhood $V_0$ of $b_0$. However, $\bR_0(f_b)$ cannot land whenever $co_b \in \bR_{1/2}(f_b)$ and $b \in \hat\cR^{\bfB}_{1/2}$. This is a contradiction. Since the set of parabolic parabolic parameters is a discrete set, this implies that $\hat\cR^{\bfB}_{1/2}$ lands at $\{b_0, -b_0\}$.

Note that $f_{b_0}$ must be separable, as $\bR_0(f_{b_0})$ has a non-trivial wake arc. By \propref{prop.separate}, the external rays bounding the wake must be fixed, and hence, their external angles must be $0$ and $1/2$.

An argument analogous to the one given in the proof of Theorem 4.1 in \cite{Mi} shows that there exists a smooth path $\gamma$ extending from $b_0$ such that for $b \in \gamma \setminus \{b_0\}$, the fixed bubble ray and the fixed external rays of $f_b$ persistently co-land at a repelling fixed point. By \lemref{fund wakes}, we have $\gamma \setminus \{b_0\} \subset \cW_{\frac{1}{6}, \frac{1}{3}} \cup \cW_{\frac{2}{3}, \frac{5}{6}}$, and $b_0 = b_*$ or $-b_*$.
\end{proof}

\begin{defn}[Fundamental region]\label{def.fundamentalregion}
The {\it fundamental region} $\cS_{\fun}$ is the connected component of
$$
\cU_\infty \setminus \left(\cW_{\frac{1}{6}, \frac{1}{3}} \cup \cW_{\frac{2}{3}, \frac{5}{6}} \cup \hat\cR^{\bfB}_{1/2} \cup \{b_*, -b_*\}\right)
$$
containing $\cV_{\frac{5}{6}, \frac{1}{6}}$ (and in particular, $\cR^\infty_0$). See Figure \ref{fig.fundamental}.
\end{defn}

We will need the following lemma in \S \ref{subsec.recurrentcase}.
\begin{lem}\label{lem.holo.motion}
   Fix $b_0\in \mathcal{S}_{\mathrm{fun}}$ and $r>1$. Set
   $$
   \mathcal{E}^\infty_{<r} := \bigcup_{1<r'<r}\mathcal{E}^\infty_{r'};
   $$
  and
  $$
  I_0^\infty := R^\infty_0(f_{b_0}) \cup R^\infty_{1/3}(f_{b_0})\cup E^\infty_r(f_{b_0})
  \matsp{and}
    I_0^\bB:= R^\bB_0(f_{b_0})\cup R^\bB_{1/2}(f_{b_0}).
  $$
Then
$$
H:(\mathcal{S}_{\mathrm{fun}}\cap\mathcal{E}^\infty_{<r})\times I_0 \longrightarrow \C
$$
given by
  $$
  H(b,z) := \left\{
\begin{array}{cl}
(\phi^\infty_b)^{-1}\circ\phi^\infty_{b_0}(z) &: z \in I_0^\infty\\
\psi_b^{-1}\circ\psi_{b_0}(z) &: z \in I_0^\bB.
\end{array}
\right.
  $$
is a holomorphic motion.
\end{lem}

\begin{proof}
    By Lemma \ref{fund wakes} \romannumeral2), the prescribed holomorphic motion exists at least in a small neighborhood of some $b_0'\in \mathcal{V}_{\frac{5}{6},\frac{1}{6}}$. It can be extended by analytic continuation and will stop only when the dynamical objects crash on the free critical point $b$. Thus the motion can be extended at least to $\mathcal{S}_{\mathrm{fun}}\cap\mathcal{E}^\infty_{<r}$. 
\end{proof}

\begin{figure}
\centering 
\includegraphics[width=0.6\textwidth]{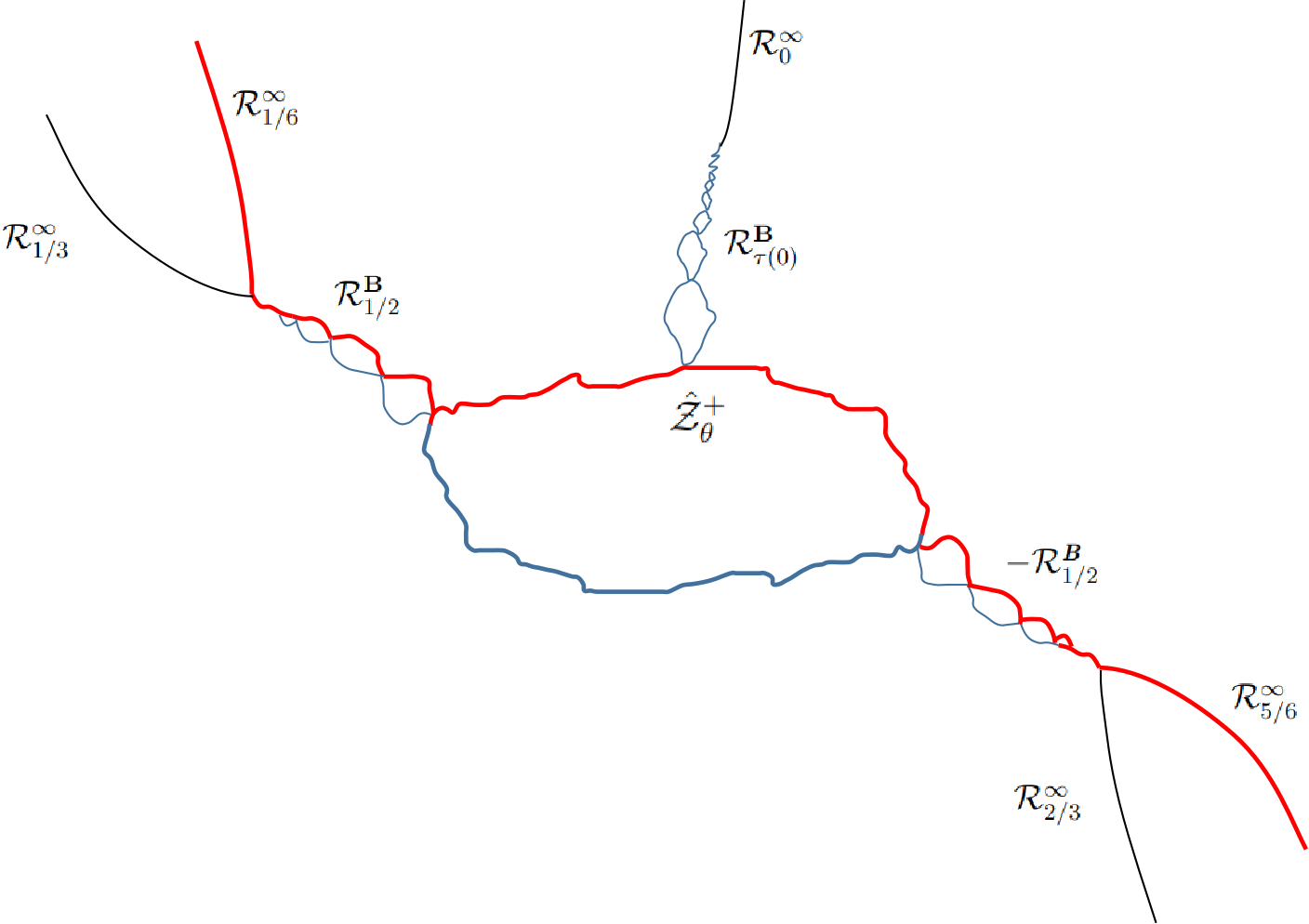} 
\caption{A schematic picture of $\mathcal{S}_{\mathrm{fun}}$. Its boundary is marked out by the red curve.} 
\label{fig.fundamental} 
\end{figure}

Let $\hat\cR^{\bfB}_t$ with $t\neq 0$ be a parameter bubble dual-ray. The set
\begin{equation}\label{eq.parabubray}
\cR^{\bfB}_t := \hat\cR^{\bfB}_t\cap \cS_{\fun}
\end{equation}
is called a {\it parabubble ray}. We say that $\cR^{\bfB}_t$ is {\it periodic} or {\it pre-periodic} if the angle $t+1/2$ is periodic or pre-periodic respectively under the angle doubling map. Additionally, we say that $\cR^{\bfB}_t$ {\it lands at $x \in \cS_{\fun}$} if $\hat\cR^{\bfB}_t$ does.

\begin{prop}\label{per para ray}
Every periodic parabubble ray lands at a parabolic parameter.
\end{prop}

\begin{proof}
The proof is analogous to the one given in \lemref{lem.fund region}, and hence, will be omitted.
\end{proof}

\begin{defn}\label{def.renor-parawake}
    A parameter wake whose based point is the landing point of a periodic parabubble ray is referred to as a {\it renormalizable wake}. The terminology is justified by the following result.

\end{defn}

\begin{prop}\label{para separable}
Let $b_0 \in \cS_{\fun} \cap \cC_\infty(\theta)$. Then $f_{b_0}$ is separable if and only if $b_0$ is contained in a renormalizable wake, or is equal to its base point, which is a parabolic parameter.
\end{prop}

\begin{proof}
Suppose that in the dynamical plane of $f_{b_0}$, the critical point $b_0$ is contained in a dynamical wake $W_{s,t}(f_{b_0})$ bounded between the external rays $R^\infty_s(f_{b_0})$ and $R^\infty_t(f_{b_0})$, which co-land at the landing point $x_{b_0}$ of some $p$-periodic bubble ray $\bR(f_{b_0})$. If $x_{b_0}$ is repelling, then the above objects persist in the dynamical plane of $f_b$ for all $b$ near $b_0$, and $b \in W_{s,t}(f_b)$.

Note that the critical point $b$ is contained in $R^\infty_s(f_b)$ in the dynamical plane for $f_b$ if and only if the parameter $b$ is contained in $\cR^\infty_{s'_i}$, with $s'_i := s+ i/3$ for some $i \in \{1, 2\}$. Note that only one of $\cR^\infty_{s'_1}$ or $\cR^\infty_{s'_2}$ is contained in $\cS_{\fun}$. Analogous statement holds for the external angle $t$. Let $i \in \{1,2\}$ and $j\in \{1,2\}$ be the indices so that $ \cR^\infty_{s'_i} \cup \cR^\infty_{t'_j} \subset \cS_{\fun}$. Arguing as in the proof of Theorem 3.1 in \cite{Mi}, we conclude that $ \cR^\infty_{s'_i}$ and $\cR^\infty_{t'_j}$ co-land at some parabolic parameter $\hat b$ to cut out a parameter wake $\cW_{s'_i, t'_j}$ containing all parameters $b$ for which there exists a dynamical wake $W_{s,t}(f_b) \ni b$ based at the landing point $x_b$ of $\bR(f_b)$, and $x_b$ is repelling.

Let $\tilde R^{\bB}(f_{b_0})$ be the unique non $f_{b_0}$-periodic bubble ray for $f_{b_0}$ such that
$$
f_{b_0}(\tilde R^{\bB}(f_{b_0})) = f_{b_0}(\bR(f_{b_0})).
$$
Consider the parabubble ray $\cR^{\bfB} := \Psi^{-1} \circ \psi_{b_0}(\tilde R^{\bB}(f_{b_0}))$. We claim that $\cR^{\bfB}$ lands at the parabolic parameter $\hat b$.

 In the dynamical plane of $f_{\hat b}$, the bubble ray $R^{\bB}(f_{\hat b})$, and the external rays $R^{\infty}_s(f_{\hat b})$ and $R^{\infty}_t (f_{\hat b})$ coland at the parabolic periodic point $z_{\hat b}$. Notice that $z_{\hat b}$ is a parabolic fixed point of $f_{\hat b}^p$ with multiplier one and it is non degenerated, i.e. we have Taylor expansion:
 $$f_{\hat b}^p(z)-z_{\hat b} = z+\omega z^2+O(z^3),\,\,\omega\not=0.$$
 This is because that $f_{\hat b}^p$ only has two critical orbit, and one of then is contained in the Siegel disk boundary $\partial D(f_{\hat b})$.
 
 Thus the "Orbit Correspondence" argument (see \cite{Ta}) can be applied. When one perturbs $\hat b$ to $b$ in a certain region (where the parabolic implosion occurs), the perturbed Fatou petal $P_b$ can be defined for $f_b^p$ and varies continuously. Moreover $\overline{P_b}$ converges to the closure of the union of the attracting and repelling petal for $f_{\hat b}^p$. Let us fix any bubble $B_{\hat b}\subset {R}^\bB(f_{\hat b})$. Denote by $\check{R}^\bB(f_{\hat b})\subset R^\bB(f_{\hat b})$ the finite sub-bubble ray consisting of all bubbles with generation less than $\mathrm{gen}(B_{\hat b})$. Then for $b$ close enough to $\hat b$, $\check{R}^\bB(f_{\hat b})$ moves holomorphically under $\psi_{b}^{-1}\circ \psi_{\hat b}$ to $\check{R}^\bB(f_{b})$ (recall Theorem \ref{cubic tree model}). Set $B_b = \psi_{b}^{-1}\circ \psi_{\hat b}(B_{\hat b})$. Therefore $\overline{B_b}$ will enter $P_b$ for $b$ close to $\hat b$. By \cite[Prop. 2.2]{Ta}, there exists $b'$ such that $b'$ is shooted to $\overline{B_{b'}}$ under the action of $f_{b'}^p$. This implies that $\mathcal{R}^\bB$ lands at $\hat b$.
\end{proof}

\begin{prop}
Every periodic parabubble ray lands at the base of a renormalizable wake.
\end{prop}

\begin{proof}
The proof is analogous to the one given in \lemref{lem.fund region} (using \propref{para separable} instead of \lemref{fund wakes}), and hence will be omitted.
\end{proof}

\begin{prop}\label{pre per para ray}
Every strictly pre-periodic parabubble ray $\cR^{\bfB}_s$ lands at a Misiurewicz parameter $b_0$. Moreover, if $t\in \bbQ/\bbZ$ is the external angle of the unique dynamical external ray for $f_{b_0}$ that lands at the co-critical point $co_{b_0}$, then the parameter external ray $\cR^\infty_t$ lands at $b_0$.
\end{prop}

\begin{proof}
Let $\bR_{s+1/2}(\q)$ be a strictly pre-periodic bubble ray for $\q$ for some $s \in \bbQ/\bbZ$. Note that $\bR_{s+1/2}(f_b) := \psi_b^{-1}(\bR_{s+1/2}(\q))$ is a well-defined bubble ray for $f_b$ landing at a pre-periodic point $x_b$ as long as $\psi_b(b)$ is not contained in either $\bR_{s+1/2}(\q)$ or an iterated image of $\bR_{s+1/2}(\q)$ under $\q$. Moreover, by \propref{para separable}, $x_b$ is an iterated preimage of a repelling point $y_b$ unless the parameter $b$ is the landing point of some periodic parabubble ray.

Let $b_0$ be a parameter in the accumulation set of the parameter bubble ray $\cR_s^{\bfB}$. Then $b_0$ is disjoint from the closure of any other parameter bubble ray (otherwise, it would disconnect the escape locus). Hence, in the dynamical plane for $f_{b_0}$, the strictly pre-periodic bubble ray $\bR_{s+1/2}(f_{b_0})$ lands at an iterated pre-image $x_{b_0}$ of a repelling point $y_{b_0}$. Moreover, this persists for all parameters $b$ near $b_0$. Consider a sequence $\{b_i\}_{i=1}^\infty$ converging to $b_0$ such that $b_i$ is contained in a parabubble in $\cR^{\bfB}_s$. It follows by continuity that $co_{b_0}$ coincides with the landing point $z_{b_0}$ of $\bR_s(f_{b_0})$. Thus, $\cR^{\bfB}_s$ must land at $b_0$.

The dynamical external ray $R^\infty_t(f_{b_0})$ for $f_{b_0}$ lands at $z_{b_0} = co_{b_0}$. Since $z_{b_0}$ is an iterated preimage of a repelling point $y_{b_0}$, it follows that for all $b$ near $b_0$, the ray $R^\infty_t(f_b)$ lands at the iterated preimage $z_b$ of $y_b$. The fact that the parameter external ray $\cR^\infty_t$ lands at $b_0$ follows immediately.
\end{proof}

\section{Combinatorial tools}\label{sec:combinatorial-tools} 

In this section, we recall two combinatorial tools as well as their a priori bounds, which will be used later in the proof of the Rigidity Theorem (Theorem \ref{thm.rigidity}). The first tool is puzzle disks constructed by Yang \cite{Y}. Its a priori bound (Proposition \ref{prop.puzzle.disk} (\romannumeral3)) controls the geometry of the Julia set near the Siegel disk boundary. The second tool is Modified principal nest by Kahn and Lyubich \cite{KaLy} (or favorite nest by Avila-Kahn-Lyubich-Shen \cite{AvKaLySh}). Its a priori bound controls the geometry near the free critical point. The application of the second tool is almost identical to \cite{AvKaLySh}. The application of the puzzle disks is not immediate, since its combinatorics is far more complicated than that of the classical Yoccoz puzzle pieces and thus the tableau rule does not hold for puzzle disks. We will prove in \S \ref{subsec.modified-puzzle-disks} that the first hit map to puzzle disks has bounded degree. This fact plays an important role in the proof of the Rigidity Theorem.

\subsection{Puzzle disks}\label{subsec.puzzle-disks}
\text{ }\\
In \cite{Y}, Yang constructs "puzzle disks" for the associated Blaschke product, and shows that the puzzle disks along the unit circle shrink to a single point. \thmref{yang lc} then follows by a quasi-conformal surgery argument (a generalization of Proposition \ref{prop.surgery} to arbitrary degree).

To simplify the presentation, we recall some important results of \cite{Y} restricting to our special case: the cubic Siegel polynomials. From now on til the end of the article, we always make the following assumption on $f:=f_b$ if there is no exception:

\begin{assum}\label{assum.well-defined}
$b\in \mathcal{S}_{\mathrm{fun}}\cap\partial{\mathcal{C}}(\theta)$ and satisfies 
\begin{itemize}
    \item[(\romannumeral1)] $b$ is non-separable.
    \item [(\romannumeral2)]$f^n(b)\not\in\overline{{R^\bB_0(f)}}$ for all $n\geq0$.
\end{itemize}
\end{assum}

In particular, if $b$ satisfies Assumption \ref{assum.well-defined}, then $R^\bB_\tau(f)$ is well-defined for all $\tau\in\mathscr{T}$ (recall $\mathscr{T}$ at the end of \S \ref{sec:quadratic Siegel}). 

\begin{lem}\label{lem.loc.alpha}
    Let $f$ satisfy Assumption \ref{assum.well-defined}, then the wake impression $\mathrm{Imp}_W(R^\bB_0)$ is trivial, i.e. equals $\{\alpha\}$, the landing point of $R^\bB_0$.
\end{lem}
\begin{proof}
    If $\mathrm{Imp}_W(R^\bB_0)$ is non trivial, then there will be an external ray $R^\infty_t(f)$ with $t\not = 0$ accumulating to $\mathrm{Imp}_W(R^\bB_0)$. However since $f$ is non-separable, the wake arc $\mathrm{Arc}_W(R^\bB_0)$ is trivial. This leads to a contradiction.
\end{proof}

Let $\textbf{R}^\infty$ be the collection of all the external rays landing at the end point of $R^\bB_0$ and $R^\bB_{1/2}$.

Fix $r>1$, define the dynamical graph of depth $n\geq0$ by
\begin{equation}\label{eq.dyna.graph}
    I_0:= R^\bB_0\cup R^\bB_{1/2}\cup \overline{\textbf{R}^\infty}\cup E^\infty_r,\,\, I_n:={f}^{-n}(I_0),\,\,I_\infty := \bigcup_{k\geq0}I_k,
\end{equation}
where $E^\infty_r$ is the equipotential of level $r$ in the immediate basin of $\infty$. For convenience, we also set 
\begin{equation}\label{eq.equipotential>r}
    E^\infty_{>r} := \bigcup_{r'>r} E^\infty_{r'}. 
\end{equation}

A {\it puzzle piece} $P_n$ of depth $n$ is a connected component of $\C\setminus I_n$. The local connectivity of $J(f)$ at $z\in\partial D(f)$ (Theorem \ref{yang lc}) follows directly from the following shrinking property for puzzle pieces attached at $\partial D(f)$:
\begin{thm}{\cite[Cor. 8.10]{Y}}\label{thm.jonguk}
    Pick any $z\in\partial D(f)$. If $z$ is a joint, denote by $P^+_n(z)$ (resp. $P^-_n(z)$) the right-hand side (resp. left-hand side) puzzle piece such that $z\in \partial P^+_n(z)$ (resp. $z\in \partial P^-_n(z)$). If not, denote by $P_n(z)$ the unique puzzle piece such that $z\in\partial P_n(z)$. For the first case,
    $\bigcap_{n\geq0}\overline{P^+_{n}(z)} = \{z\}$, $\bigcap_{n\geq0}\overline{P^-_{n}(z)} = \{z\}$; for the second case $\bigcap_{n\geq0}\overline{P_{n}(z)} = \{z\}$. In particular, the Julia set $J(f)$ is locally connected at $\bigcup_{k\geq0}f^{-k}(\partial D(f))$.
\end{thm}

The two following Corollaries follows from a standard argument, see \cite{Yam}.

\begin{cor}\label{cor.landing.two.external}
    Let $B$ be a bubble of $f$ of generation $n\geq 1$, $x$ be its root. Then there are exactly two external rays landing at $x$.
\end{cor}

We denote by $W_x$ the wake cut by the two external ray in Corollary \ref{cor.landing.two.external}. Denote by $L_x := \overline{W_x}\cap K(f)$ the corresponding {\it limb}. 

\begin{cor}\label{cor.decomposition}
    Let $B,x$ be as in Corollary \ref{cor.landing.two.external}. We have decompositions $K(f) = \overline{D(f)}\cup \bigcup_{x'} L_{x'}$ where $x'$ runs over the joints on $\partial D(f)$; $L_x = \overline{B}\cup \bigcup_{x''} L_{x''}$, where $x''$ runs over the joints on $\partial D(f)\setminus\{x\}$.
\end{cor}

In order to construct puzzle disks, we need to work with the Blaschke product $F :=Bl$ associated to $f$ (Proposition \ref{prop.surgery}). Recall that $F$ is quasiconformaly conjugate to $f$ on $\C\setminus\overline{\D}$ and is conjugate to it self by $\kappa: z\mapsto1/\overline{z}$. Let $c_b$ be the unique critical point of $F$ in $\C\setminus\overline{\D}$. Denote $g := F|_{\partial\D}$. Notice that $g:\partial\D\to\partial\D$ is an analytic homeomorphism with a unique critical point $z=1$ and $g$ has bounded type rotation number 
$$\theta = [a_1,a_2,...,a_n,...] := \cfrac{1}{a_1+\cfrac{1}{a_2+\cfrac{1}{a_3+...}}},\quad a_n\text{ is uniformly bounded from above}.$$
Denote by $p_n/q_n = [a_1,...a_n]$ the $n$-th approximant of $\theta$, where $q_n$ is called the $n$-th closest return moment. It is classical that $q_n$ satisfies $q_n = a_{n-1}q_{n-1} +q_{n-2}$. The following lemma is an elementary exercise.
\begin{lem}\label{lem.property-q_n}
    Set $r_n = q_{n}+q_{n+1}$, $\mathbf{r}_n = \Sigma_{i=1}^{n}r_n$. For $n\geq3$, we have 
    $$q_{n+2}\geq r_{n}>\mathbf{r}_{n-2}.$$
\end{lem}
For $k\in\mathbb{Z}$, denote $c_k := g^k(1)$. For $x,y\in \partial\D$ such that $x\not=-y$, denote by $(x,y),[x,y]$ the unique open/closed interval with end points $x,y$ and with length less than $\pi$. For $m,n\in\mathbb{Z}$, $x = g^m(1)$, $y = g^n(1)$, define $(m,n)_c := (x,y)$, $[m,n]_c := [x,y]$. Denote $J^\pm_n := ({\pm q_n},{\pm q_{n+1}})_c$. By Swi\c{a}tek's real a priori bound \cite{Sw}, $g$ is quasi-symmetrically conjugate to the rigid rotation. In particular, we have the following obvious relation:
\begin{equation}\label{eq.JnJn+1}
    J_{n+1}^\pm\subset J_{n}^\pm,\,\,J_{n+2}^\pm\Subset J_{n}^\pm,\,\,J_{n+1}^+\Subset J_{n}^-.
\end{equation}

By Proposition \ref{prop.surgery}, $f_b|_{\bbC\setminus D(f_b)}$ is quasiconformally conjugate to a Blaschke product $Bl$ by $\varphi$. Define
$$
J_1(Bl) := J(Bl)\setminus(\cup_{n\geq0}Bl^{-n}(\D))=\varphi(J(f_b))
$$
Denote the reflection about $\partial\D$ by $\kappa:z\mapsto1/\overline{z}$.
An external bubble ray $R^{\bB}(Bl)$ is defined to be $\varphi(R^{\bB}(f_b)\setminus{D(f_b)})$.

Let ${D'}$ be the unique connected component in $\C\setminus\D$ of $F^{-1}(\D)$ such that $1\in \partial D'$. The $n$-th backward dynamical partition of $\partial\D$ and $\partial D'$  are defined to be
\begin{equation}\label{eq.dynamical.partition}
        \mathcal{T}_n := \{g^{-i}(1);\,0\leq i<q_{n}+q_{n+1}\},\,\,\mathcal{T}_n' := (F|_{\partial D'})^{-1}(F(\mathcal{T}_n)).
\end{equation}

Let $\textbf{R}^\infty(F)$ be the collection of all the external rays landing at the end point of $R^\bB_0(F)$ and $R^\bB_{1/2}(F)$. Fix $r>1$, define the dynamical graph of depth $n\geq0$ by
\begin{equation}\label{eq.dyna.graph.Blaschke}
    I_0(F):= R^\bB_0(F)\cup R^\bB_{1/2}(F)\cup \overline{\textbf{R}^\infty(F)}\cup E^\infty_r(F),\,\, I_n(F):=({F}|_{\C\setminus\D})^{-n}(I_0(F)),
\end{equation}
where $E^\infty_r(F)$ is the equipotential of potential $r$ in the immediate basin of $\infty$. The {\it puzzle piece} $P^F_n(z)$ of depth $n$ containing $z$ is the connected component of $\C\setminus I_n(F)$ containing $z$. Denote by $P^{F,-}_n(1),P^{F,+}_n(1)\subset\C\setminus\overline{\D}$ the left/right-hand side puzzle piece containing $z=1$ on its boundary. 

\begin{figure}[ht]
\centering 
\includegraphics[width=0.85\textwidth]{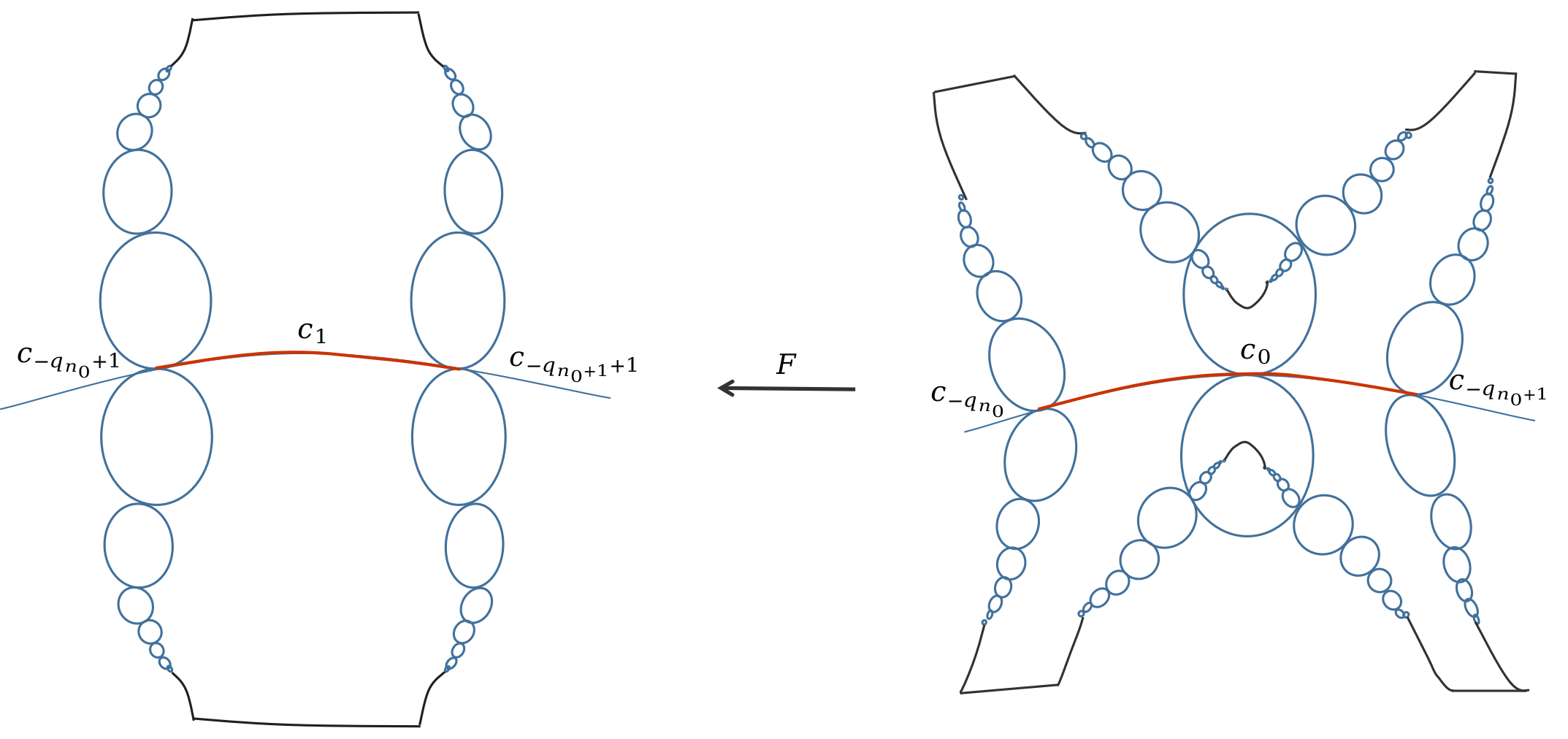} 
\caption{A schematic picture of the map $F|_{D^{n_0}}$, with its image being the puzzle silhouette $S_{Jv_{n_0}}$ based on $(-q_{n_0}+1,-q_{n_0+1}+1)_c$. The figure on right is $D^{n_0}$; the figure on the left is the puzzle silhouette $S_{Jv_{n_0}}$.} 
\label{fig.Dn0} 
\end{figure}
\begin{figure}[ht]
\centering 
\includegraphics[width=0.95\textwidth]{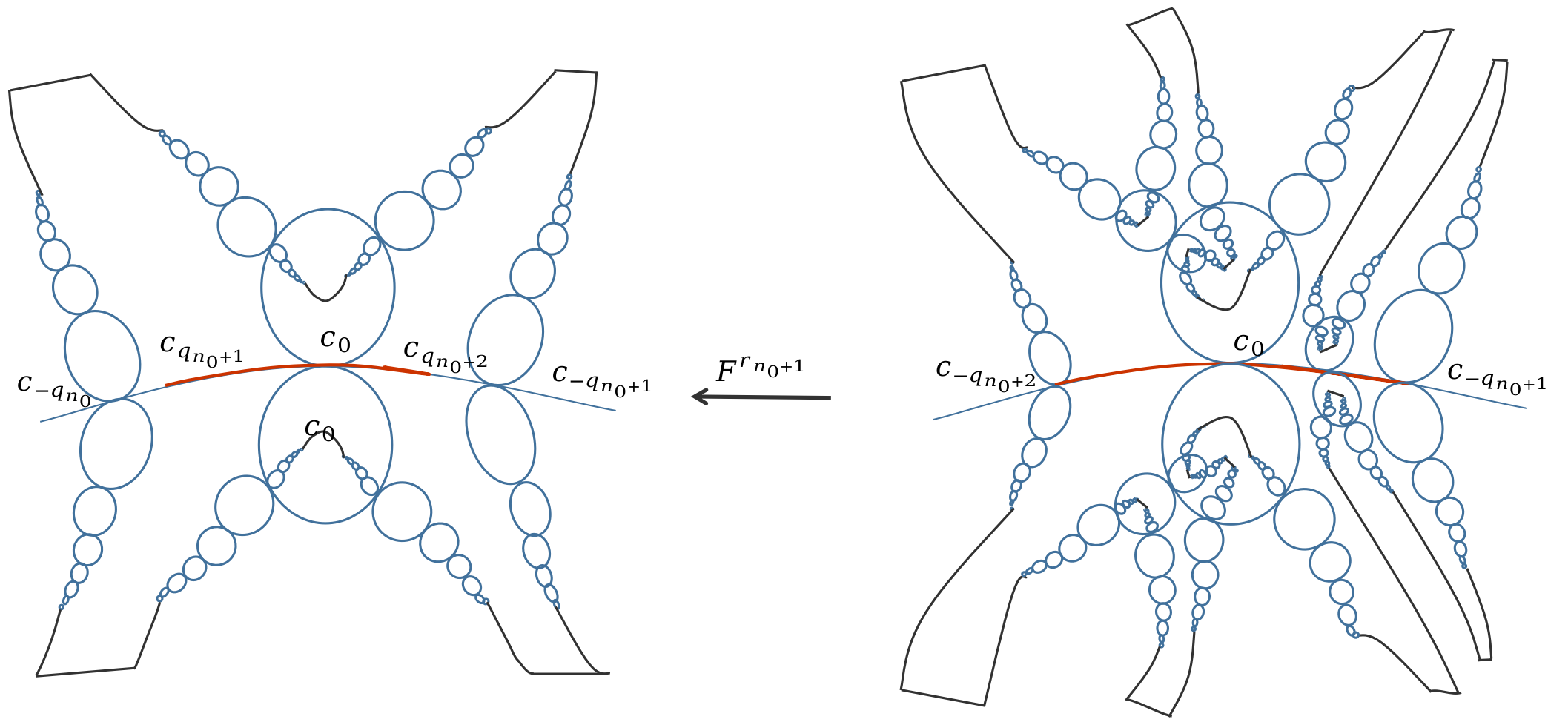} 
\caption{A schematic picture of the map $F^{r_{n_0+1}}:D^{n_0+1}\to D^{n_0}|_{J_{n_0+1}^+}$.}  
\label{fig.Dn0+1} 
\end{figure}
\begin{figure}[ht]
\centering 
\includegraphics[width=0.9\textwidth]{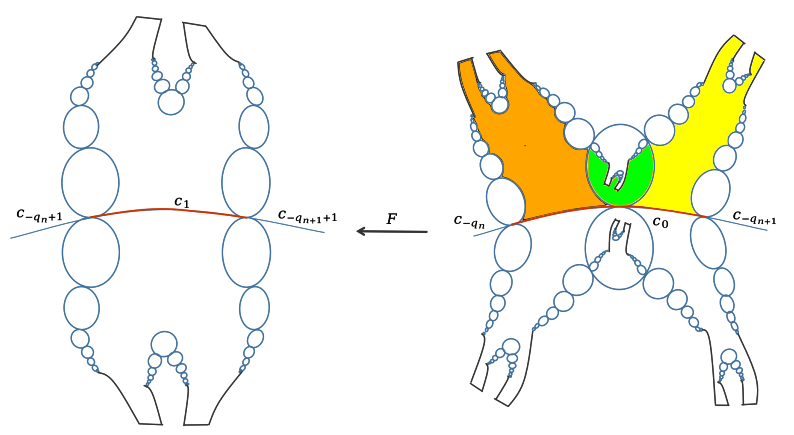} 
\caption{The construction of the modified puzzle disk $C^n$. The figure on the right is $C^n$: the regions $Q^{n}_-,Q^{n},Q^{n}_+$ are marked out by color orange, green and yellow respectively. The figure on the left is the topological disk $S^n$.}  
\label{fig.Cn} 
\end{figure}

\begin{defn}[Puzzle silhouette]\label{def.puzzle-silhouette}
    For any $n\geq 0$, let $r_n = r/3^{q_{n+1}-1}$. The {\it puzzle silhouette} based on $Jv_n := (-q_n+1,-q_{n+1}+1)_c$, denoted by $S_{Jv_n}$, is defined to be the connected component of $$\C\setminus(\mathbf{V}_n\cup E^\infty_{r_n}\cup\kappa(\mathbf{V}_{{n}}\cup E^\infty_{r_n}))$$
    such that $({-q_{n}+1},{-q_{n+1}+1})_c\subset S_{Jv_n}$, where $\mathbf{V}_{n}$ is the union of all bubble rays in $I_{q_{n+1}-1}(F)$ rooted at $c_{-q_{n}+1}$ and $c_{-q_{n+1}+1}$.
\end{defn}

We fix $n_0$ large enough, such that
\begin{itemize}
    \item $({-q_{n_0}},{-q_{n_0+1}})_c\cap({-q_{n_0}+1},{-q_{n_0+1}+1})_c = \emptyset$.
    \item the puzzle silhouette based on $({-q_{n_0}+1},{-q_{n_0+1}+1})_c$ does not contain $F(c_b)$.
\end{itemize}
\begin{defn}[Puzzle disks]\label{def.puzzle-disks}
    The {\it puzzle disk} $D^{n_0}$ of the initial level $n_0$ is defined to be the connected component of $F^{-1}(S_{Jv_{n_0}})$ containing $w=1$. For $n>n_0$, the puzzle disk $D^n$ of level $n>n_0$ is constructed by pulling back $D^{n-1}$ along $\partial\D$ with two slits cut out in $\partial\D$: define $D^{n}$ to be the connected component of $F^{-r_n}({D}^{n-1}|_{J_n^+})$ that intersects with $\partial\mathbb{D}$, where ${D}^{n-1}|_{J_n^+} := D^{n-1}\setminus(J_{n-1}^-\setminus \overline{J_n^+})$.
\end{defn}

\begin{lem}\label{lem.no-detached-bubble}
    There exists $n_0$ large enough, such that for all $n\geq n_0+1$, the orbit $$D^n\to F(D^{n})\to...\to D^{n-1}|_{J_n^+}$$
    does not contain the free critical value $F(c_b)$. Symmetrically it does not contain $\kappa(F(c_b))$. In particular, $\partial D^n$ does not intersect the closure of any bubble in $\C\setminus\overline{\D}$ that is not connected to $\D$ by a bubble ray.
\end{lem}
\begin{proof}
 Since $F$ and $f$ are quasiconformally conjugate on $\C\setminus D(f)$, the wakes $W^F_x$ and limbs $L^F_x$ in $\C\setminus\D$ for $F$ are naturally defined, and one has the parallel version of Corollary \ref{cor.decomposition} for $F$. Hence $F(c_b)$ is contained in two nested limbs $L^F_{x_1}\Subset L^F_{x_2}$, such that $x_1\in\partial\D$ being a preimage of $1$ and $x_2\in\partial B_{x_1}\setminus\{x_1\}$ also being a preimage of $1$, where $B_{x_1}\subset\C\setminus\overline{\D}$ is the bubble rooted at $x_1$. From the construction of $D^n$, it is easy to see that $$D^{n-1}\cap \partial\D=J^{-}_{n-1},\,\, F(D^{n-1}\cap\partial D') = F(J^-_{n-1}).$$
    Let us assume that $F^{s_1}(x_1) = 1$ and $F^{s_2}(F^{l}(x_2)) = 1$, where $l\geq 1$ is the smallest integer such that $F^l(x_2)\in\partial\D$. We take $n_0$ large enough, such that $q_{n_0}>\max\{s_1,s_2\}$. Thus $x_1\in\mathcal{T}_{n_0}$ and $F^{l-1}(x_2)\in\mathcal{T}'_{n_0}$ (recall (\ref{eq.dynamical.partition})). Moreover $F^{l-1}(x_2)\not\in D^{n}\cap\partial\D$. The Lemma then follows easily.
\end{proof}

From now on we always fix such $n_0$ in Lemma \ref{lem.no-detached-bubble}. We summarize the main results in \cite[\S 5]{Y} and \cite[Thm 7.1]{Y} as follows:
\begin{prop}\label{prop.puzzle.disk}
    The puzzle disk $D^n$, $n\geq n_0$, is a topological open disk containing $1$ such that
    \begin{enumerate}
        \item[$(\mathrm{\romannumeral1})$] $\kappa D^n = D^n$, where $\kappa:w\mapsto1/\overline{w}$; $D^{n}\cap \partial\D=J^{-}_n$; $F(D^n\cap\partial D') = F(J^-_n)$. 
        \item[$(\mathrm{\romannumeral2})$]  $F^{r_n}:D^n\to D^{n-1}|_{J^+_n}$ is a proper map with uniformly bounded degree (not depending on $n$).
        \item[$(\mathrm{\romannumeral3})$]   $D^{n+1}\subset D^n$, $D^{n+2}\Subset D^n$, $\mathrm{mod}(D^n\setminus\overline{D^{n+2}})>\mu$ for some uniform constant $\mu>0$ not depending on $n$.
    \end{enumerate}
\end{prop}

The combinatorics of puzzle disks $D^n$ is not compatible with the puzzle pieces $P_m^F(z)$: as $n$ grows, $D^n$ will develop more and more preimages of $P_m^F(z)$ inside bubbles of $F$ (see Figure \ref{fig.Dn0+1}). We introduce the notion of {\it modified puzzle disks}, which has a much simpler combinatorics.

\begin{defn}[Modified puzzle disks]\label{def.modified-disks}
    Fix $n_0$ to be the initial level of puzzle disks. For $n\geq n_0$, let $P^n\subset\C\setminus\D$ be the puzzle piece of depth $q_{n+1}-1$, such that $\partial P^n\cap\partial\D = Jv_n$. Define a topological disk $S^n = P^n\cup Jv_n\cup\kappa P^n$ (notice that 
    $F(1)\in S^n$). The modified puzzle disk $C^n$ is defined to be the connected component of $F^{-1}(S^n)$ that contains $w=1$. See Figure \ref{fig.Cn}.
\end{defn}

The following Lemma follows immediately from the construction of $C^n$:
\begin{lem}\label{lem.base.Cn}
    $\partial C^n\cap\partial\D = J_n^-$, $\partial C^n\cap\partial D' = F(J_n^-)$. In particular, $D^n\cap\partial\D = C^n\cap\partial\D$, $D^n\cap\partial D' = C^n\cap\partial D'$.
\end{lem}

From Lemma \ref{lem.base.Cn}, it is obvious that $[C^n\cap(\C\setminus\overline{\D})]\setminus(C^n\cap\partial D')]$ has exactly three connected components: $Q^n_-,Q^n,Q^n_+$, such that $Q^n_-$ (resp. $Q^n_+$) is on the left-hand (resp. right-hand) side of 1, $Q^n\subset D'$. Notice that we have 
\begin{equation}\label{eq.Qn-are-puzzlepieces}
    Q^n_\pm = P^{F,\pm}_{q_{n+1}}(1)
\end{equation}
We have the following decomposition
\begin{equation}\label{eq.Dn-decomposition}
        C^n\cap(\C\setminus\overline{\D}) = \mathrm{int}(\overline{Q^n_-\cup Q^n\cup Q^n_+});\,\,
        F(\partial Q^n_\pm\cap \partial Q^n) = F(\partial Q^n_\mp\cap\partial\D).
\end{equation}

The following crucial lemma says that $C^n$, $D^n$ are combinatorially equivalent.


\begin{lem}\label{lem.Cn-Dn}
    $C^{n+3}\subset D^n\subset C^{n-2}$. In particular $C^{n+7}\Subset C^{n}$, $\mathrm{mod}(C^{n}\setminus\overline{C^{n+7}})>\mu$.
\end{lem}
\begin{proof}
    First we prove $C^{n+3}\subset D^n$. Suppose the contrary. By Proposition \ref{prop.puzzle.disk} (\romannumeral1), Lemma \ref{lem.base.Cn} and (\ref{eq.JnJn+1}), we have $$C^{n+3}\cap(\partial\D\cup\partial D')\Subset D^n\cap(\partial\D\cup\partial D').$$ 
    Thus there exists $w_0\in\partial D^n\cap C^{n+3}$, $w_0\not\in\partial\D\cup\partial D'\cup\kappa(\partial D')$. By symmetry, we may assume that $w_0\in\C\setminus(\overline{\D\cup D'})$. From the construction of $C^{n}$, $F(w_0)$ belongs to a puzzle piece $P_{q_{n+4}-1}(F(w_0))$ of depth $q_{n+4}-1$. Hence $F^{q_{n+4}}(w_0)\not\in I_0(F)$. On the other hand, from the construction of $D^n$ and Proposition \ref{prop.puzzle.disk} (\romannumeral2), $F^{\mathbf{r}_n}(w_0)\in \partial I_0(F)\cup\partial(\kappa I_0(F))$. Notice that $\partial I_0(F)\cup\partial(\kappa I_0(F))$ is forward invariant by $F$. Thus for all $n\geq\mathbf{r}_n$, $F^n(w_0)\in\partial I_0(F)\cup\partial(\kappa I_0(F))$. However $q_{n+4}>\mathbf{r}_n$ by Lemma \ref{lem.property-q_n}. This leads to a contradiction.

    Next we prove $D^n\subset C^{n-2}$. Suppose the contrary. Similarly as above, we have $$D^{n}\cap(\partial\D\cup\partial D')\Subset C^{n-2}\cap(\partial\D\cup\partial D').$$ 
    Thus there exists $w_0\in\partial C^{n-2}\cap D^{n}$, $w_0\not\in\partial\D\cup\partial D'\cup\kappa(\partial D')$. By symmetry, we may assume that $w_0\in\C\setminus(\overline{\D\cup D'})$. Without loss of generality, we assume further more that $w_0\in \partial Q_+^{n-2}$.
    
    From the construction of $C^n$, $F^{q_{n}}(w_0)$ belongs to $\partial I_0(F)$. On the other hand, from the construction of $D^n$ and Proposition \ref{prop.puzzle.disk} (\romannumeral2), $F^{{r}_n}(w_0)\in D^{n-1}|_{J_n^+}$. Recall that $r_n = q_n+q_{n+1}$. Hence $F^{r_n}(w_0)\in\partial\D\cup\partial D'$ since $D^{n-1}\cap\partial I_0(F)\subset \partial\D\cup\partial D'$. Now if $w_0\in\partial C^{n-2}$ belongs to some external ray or some equipotential of $F$, then we get a contradiction. If not, then $w_0$ belongs to $\partial B$, where $B\subset\C\setminus\overline{\D}$ is some bubble of $F$ such that $F^{q_n}(B)\subset I_0(F)$. Then there exists an infinite bubble chain $(B=)B_0,B_1,...,B_i,...$ such that $\partial B_i\cap\partial C^n\not = \emptyset$ and $\bigcup \overline{B_i}$
    colands with an external ray $R^\infty_t(F)$ that also intersects $\partial C^{n-2}$. 
    
    We claim that $B_0$ cannot be connected to $\D$ by finitely many bubbles. Indeed, otherwise it is connected to $D'$ or $\D$ by bubbles contained in one of the two bubble rays that bound $Q^{n-2}_+$ (See Figure \ref{fig.Cn}). Hence $\overline{B_0}$ does not intersect $D^n$, a contradiction. Thus the claim is proved. On the otherhand, by Lemma \ref{lem.no-detached-bubble}, $\partial D^n$ does not intersect the closure of any bubble that is not connected to $\partial\D$. Hence $\overline{B_0}\subset D^n$, and $\overline{B_i}\subset D^n$ for all $i$. Thus $R^\infty_t(F)\cap D^n\not = \emptyset$. Hence we can take $w'_0\in R^\infty_t(F)\cap D^n$ and we get a contradiction by applying the same argument for $w_0$ to $w_0'$.
\end{proof}

The next lemma says that nested topological disks around $\partial\D$ with a priori bounds are comparable to round disk at all scales.
\begin{lem}\label{lem.commensurable-round-disk}
    Fix any $c_l$ with $l\in\Z$. Let $U^{n+1}\subset U^{n}$ be a nested sequence of Jordan disks such that $U^n\cap\partial\D = (-q_{n+1}+l,-q_n+l)_c$, and there exists $m\geq0$, $\nu>0$, such that $U^{n+m}\Subset U^{n}$,
    $\mathrm{mod}(U^n\setminus U^{n+m})>\nu$. Let $B^n(c_l)$ be the round disk centered at $c_l$, with radius $|c_{-q_n+l}-c_l|$. Then there exists ${N} ={N}(m,\nu)$ such that 
    $$U^{n+2N}\subset B^{n+N}(c_l)\subset U^{n}.$$
\end{lem}
\begin{proof}
    We need the following result due to McMullen:
\begin{thm}[{\cite[Thm. 2.1]{Mc}}]\label{thm.mcmullen}
    Let $A\subset\C$ be an annulus and $z_0\in\C$ such that $z_0\not\in A$. Denote by $\Omega_1,\Omega_2$ the bounded/unbounded component of $\C\setminus\overline{A}$. If $A$ has sufficiently
large modulus, then it contains the round annulus $A'$ centered at $z_0$, with inner (resp. outer) radius equals $\sup_{z\in \overline{\Omega_1}}|z|$ (resp. $\inf_{z\in \overline{\Omega_2}}|z|$). Moreover, $\mathrm{mod}(A) = \mathrm{mod}(A')+O(1)$.
\end{thm}
 Thus we can take $N = N(m,\nu)$ large enough, such that $A=U^n\setminus\overline{U^{n+N}}$ has large modulus, so that it contains $A'$ described in the above Theorem. Let $\Tilde{B}$ be the inner component of $\C\setminus\overline{A'}$. Then $B^{n+N}(c_l)\subset\Tilde{B}$ since $U^{n+N}\subset\Tilde{B}$ and $U^{n+N}\cap\partial\D = (-q_{n+N+1}+l,-q_{n+N}+l)_c$. Thus $B^{n+N}\subset U^n$. On the other hand, $U^{n+N}\subset A'\cup\overline{\Tilde{B}}$. Notice that $(A'\cup\overline{\Tilde{B}})\subset B^n(c_l)$, since $U^{n}\cap\partial\D = (-q_{n+1}+l,-q_{n}+l)_c$. Thus $U^{n+N}\subset B^n(c_l)$. The result follows.
\end{proof}

We say that two sequences of bounded sets $\{K_n\},\{K'_n\}$ in $\C$ have {\it commensurable size} if there exists a constant $M>1$ such that
$$\frac{1}{M}\leq\frac{\mathrm{diam}(K_n)}{\mathrm{diam}(K_n')}\leq{M}.$$
\begin{cor}\label{cor.commensurable}
$D'\cap C^n$ contains a round disk $V^n$ that has commensurable size with $C^n$.
\end{cor}
\begin{proof}
By Proposition \ref{prop.puzzle.disk},(\romannumeral3) and Lemma \ref{lem.Cn-Dn}, we get that $C^{n}\setminus \overline{C^{n+6}}$ has modulus bounded from below. By Lemma \ref{lem.commensurable-round-disk}, there exists $N$ such that 
$$C^{n+2N}\subset B^{n+N}(1)\subset C^{n}.$$ Since $C^n\cap\partial\D = J_n^-$, we have $J^-_{n+2N} < |\gamma_n|< J^-_{n}$, where $\gamma_n:=B^{n+N}(1)\cap \partial\D$. Since $\theta$ is of bounded type, $J^-_{n+N}$ and $J^-_n$ have commensurable size by Swi\c{a}tek's real a priori bound \cite{Sw}. Thus $C^{n+N}$, $\gamma_n$ and $B^{n+N}(1)$ also have commensurable size. Since 1 is not a cusp of $\partial D'$, $B^{n+N}(1)\cap D'$ contains a round disk $V^n$ that has commensurable size with $B^{n+N}(1)$.
\end{proof}

\subsection{First hit map to puzzle disks has bounded degree}\label{subsec.modified-puzzle-disks}
\text{ }\\
Recall that we always fix some $f$ satisfying Assumption \ref{assum.well-defined} and consider its associated Blaschke product $F$, which is quasiconformally conjugate to $f$ on $\C\setminus\D$ by $\varphi$. Recall that $J_1(F) =: \varphi(J(f))$. 
\begin{lem}\label{lem.bounded.degree.puzzledisk}
    Fix any integer $m\geq0$. Let $C^{n,k}$ be the pull back of $C^n$ along $\partial\D$, (that is, the connected component of $F^{-k}(C^{n})$ that intersects with $\partial\D$) where $0\leq k\leq q_{n+m}$. Then $F^{k}:C^{n,k}\to C^n$ has bounded degree in $n$ (only depend on $m$ and $\theta$).
\end{lem}
\begin{proof}
    By Corollary \ref{cor.decomposition}, there exists a unique wake $W_x\subset \C\setminus\overline{\D}$ rooted at $x\in\partial\D$ that contains the free critical point $c_b$. Since $\theta$ is of bounded type, the orbit $C^{n,0}\cap\partial\D,C^{n_1}\cap\partial\D,...,C^{n,q_{n+m}}\cap\partial\D$ only meets $x$ and $1$ bounded many times that only depend on $m$ and $\theta$. Hence the bounded degree statement follows.
\end{proof}

\begin{lem}\label{lem.infinitesurround-cb}
    There exists an infinite sequence of nested puzzle pieces $P^F_{l_{i+1}}(c_b)\Subset P^F_{l_i}(c_b)$.
\end{lem}
\begin{proof}
    It suffices to prove that for any $l_0$, there exists $l_1>l_0$ such that $P^F_{l_1}(c_b)\Subset P^F_{l_0}(c_b)$. Notice that $c_b\in\bigcap_l\overline{P^F_{l}(c_b)}$. Thus if there does not exist such $l_1$, then 
    \begin{itemize}
        \item  either there exists $\alpha_{l_0}\in\partial P^F_{l_0}(c_b)$ which is a landing point of some bubble ray pre-periodic to $R^\bB_0(F)$, and $\alpha_{l_0}\in\partial P^F_{l}(c_b)$ for all $l$ large enough; 
        \item  or there exists a bubble $B$ such that $\partial B\cap\partial P^F_{l}(c_b)\not=\emptyset$ for all $l$ large enough.
    \end{itemize}
  The first case contradicts Lemma \ref{lem.loc.alpha}, which implies that $\bigcap_l\overline{P^F_{l}(c_b)} = \{\alpha_0\}$; the second contradicts Theorem \ref{thm.jonguk}, which implies that $\bigcap_l\overline{P^F_{l}(c_b)}\subset\partial B$. The lemma is proved. 
\end{proof}

In the sequel, we always assume that $w\in J_1(F)$ such that $\{F^n(w)\}$ accumulates to $\partial\D$ (it might also hit $\partial\D$). For any $n\geq n_0$, consider the first hit map of $w$ to $C^n$
$$F^{m_n}:C^{n,m_n} \to C^{n,0} := C^{n},$$
 where we denote by $C^{n,k}$ the $k$-the preimage of $C^{n,0}$ along the orbit. Fix any $1\leq m\leq n-n_0$, consider the pullback of $C^{n-m}$ along the orbit of $C^{n,k}$ by $F^{m_n}$, we get an orbit 
$$C^{n-m,m_n},C^{n-m,m_n-1},...,C^{n-m}.$$
Let $0\leq k_n\leq m_n$ be the largest integer such that $C^{n-m,k_n}\cap\partial\D\not=\emptyset$.
\begin{lem}\label{lem.bounded.degree.alongD}
  Suppose that there exists $N\geq0$ such that for all $l\geq 1$, $F^l(w)\not\in P^F_N(c_b)$. Then there exists a constant $M(m,\theta,N)$ such that for $n\geq M(m,\theta,N)$, the following map has uniformly bounded degree that only depends on $m$ and $\theta$:
   $$F^{k_n}:C^{n-m,k_n}\to C^{n-m}.$$
\end{lem}
\begin{proof}
The key point of the proof is to bound $k_n$. Set  $\Tilde{C}^{n} := C^{n-m}$, $\Tilde{Q}^n_\pm := Q^{n-m}_\pm$. $w_0 := F^{m_n}(w)\in\Tilde{Q}^n_+$, $w_k = F^{m_n-k}(w)$ for $0\leq k\leq m_n$. To be definite, suppose that
$$w_0\in \Tilde{Q}^{n}_+,\,\,\partial\Tilde{Q}^{n}_+\cap\partial\D = [0,-q_{n-m}]_c.$$ 
Define $\Tilde{Q}_0(+) := \Tilde{Q}^{n}_+$. Assume that for $0\leq s\leq a_{n-m+1}-1$, $\Tilde{Q}_s(+)$ is defined, and that 
$$\partial\Tilde{Q}_s(+)\cap\partial\D = [0,-sq_{n-m+1}-q_{n-m}]_c.$$
Then $(F|_{\C\setminus\overline{\D}})^{-q_{n-m+1}}(\Tilde{Q}_s(+))$ has exactly two connected components whose boundary intersects with $\partial\D$. Denote by $\Tilde{Q}_{s+1}(+),\Tilde{Q}_{s+1}(-)$ the one on the right/left-hand side of $1$. It follows that $$\partial\Tilde{Q}_{s+1}(+)\cap\partial\D = [0,-(s+1)q_{n-m+1}-q_{n-m}]_c,\,\, \partial\Tilde{Q}_{s+1}(-)\cap\partial\D = [-q_{n-m+1},0]_c.$$
If $k_n\leq a_{n-m+1}q_{n-m+1}$, then we are done. Otherwise, we make the following claim: 
\begin{claim*}
There exists a constant $M(m,\theta,N)$ such that for $n\geq M(m,\theta,N)$ and $k\leq a_{n-m+1}q_{n-m+1}$, $\Tilde{C}^{n,k}$ does not contain $c_b$. 
\end{claim*}
\begin{proof}
    Suppose the contrary, then let $k_0\leq a_{n-m+1}q_{n-m+1}$ be the smallest integer such that $F(c_b)\in\Tilde{C}^{n,k_0}$. While by Lemma \ref{lem.bounded.degree.puzzledisk}, $F^{k_0}:\Tilde{C}^{n,k_0}\to\Tilde{C}^{n}$ has bounded degree $d(m,\theta)$. Thus the number of connected components of $F^{-k_0}(\Tilde{C}^n\setminus(\partial\D\cup\partial D'))$ in $\Tilde{C}^{n,k_0}$ is bounded by a constant $N(m,\theta)$. By the non recurrent condition on $w$, there exists $N$ (not depending on $n$) such that $F^{m_n-k_0}(w)\not\in P^F_N(c_b)$. By Lemma \ref{lem.infinitesurround-cb}, $c_b$ is surrounded by $P^F_{l_{i+1}}(c_b)\Subset P^F_{l_i}(c_b)$. We may assume that $l_0 = N$. Then any path in $\Tilde{C}^{n,k_0}$ that connects $c_b$ to $F^{m_n-k_0}(w)$ passes through at least $i+1$ puzzle pieces of depth $l_i$. On the other hand, by (\ref{eq.Qn-are-puzzlepieces}), there exists a uniform constant $M(m,\theta,N)$, such that for $n\geq M(m,\theta,N)$, every component of $(F|_{\C\setminus\overline{\D}})^{-k_0}(\tilde{Q}^n_+\cup \tilde{Q}^n_-)$ in $\Tilde{C}^{n,k_0}$ is contained in a puzzle piece of depth $l_{N(m,\theta)}$. Other components of $F^{-k_0}(\tilde{C}^{n}\setminus(\partial\D\cup\partial D'))$ are contained in bubbles of $F$. Thus it leads to a contradiction if we take $i\geq N(m,\theta)$.
\end{proof}
Thus if $k_n>a_{n-m+1}q_{n-m+1}$, there exists $q_{n-m+1}\leq l_1\leq a_{n-m+1}q_{n-m+1}$ such that $w_{l_1}$ belongs to a puzzle piece $A_1$, where $A_1$ equals some $\Tilde{Q}_{s+1}(-)$ or $\Tilde{Q}_{a_{n-m+1}}(+)$. Notice that $\partial A_1\cap\partial\D = [-q_{n-m+1},0]_c$ or $[0,-q_{n-m+2}]_c$ and $F^{l_1}(A_1) = A_0$, where we set $A_0 :=\Tilde{Q}^{n}_+$. Inductively, for $i\geq 1$, we can construct $A_i$ and $l_i$ as long as 
$$k_n>\Sigma_{j=1}^{i}a_{n-m+2j-1}q_{n-m+2j-1},$$
such that
\begin{equation}\label{eq.estimate-l_i}
    \begin{aligned}
    &\partial A_i\cap\partial\D \text{ is one of }[-q_{n-m+i},0]_c,....,[0,-q_{n-m+2i}]_c;\\
    &\Sigma_{j=1}^{i}q_{n-m+j}\leq l_i\leq \Sigma_{j=1}^{i}a_{n-m+2j-1}q_{n-m+2j-1};\\
    &w_{l_{i}}\in A_i,\,\,F^{l_i}(A_i)= A_0.
\end{aligned}
\end{equation}
Notice that by (\ref{eq.estimate-l_i}) $A_{m+1}$ is a puzzle piece of depth at least $q_{n+1}$, while $Q^n_\pm$ are puzzle pieces of depth $q_{n+1}$ (recall (\ref{eq.Qn-are-puzzlepieces})), hence $A_{m+1}\subset C^{n}$. Since $F^{m_n}:C^{n,m_n}\to C^n$ is the first hit map, we must have $$k_n\leq \Sigma_{j=1}^{m+1}a_{n-m+2j-1}q_{n-m+2j-1}\leq \mathbf{r}_{n+m+2}<q_{n+m+6}.$$
Thus the lemma follows by Lemma \ref{lem.bounded.degree.puzzledisk}.
\end{proof}

\begin{lem}\label{lem.bounded.first.hit}
Suppose that there exists $N\geq0$ such that for all $l\geq 1$, $F^l(w)\not\in P^F_N(c_b)$.
There exists a constant $M(m,\theta,N)$, such that for $n\geq M(m,\theta,N)$, the orbit ${C}^{n-m,m_n-1},...,{C}^{n-m,{k}_n+1}$
does not contain $c_b$. In particular, $F^{m_n}:{C}^{n-m,m_n}\to C^{n-m}$ has uniformly bounded degree $d(m,\theta)$
\end{lem}
\begin{proof}
   One can adapt the same argument in the proof for the claim in Lemma \ref{lem.bounded.degree.alongD} to prove that the orbit ${C}^{n-m,m_n-1},...,{C}^{n-m,{k}_n+1}$ does not contain $c_b$. The bounded degree statement then follows from Lemma \ref{lem.bounded.degree.alongD}.
\end{proof}

\subsection{Favorite Nest}
\text{ }\\
We recall the classical Branner-Hubbard-Yoccoz puzzle theory. Let $\textbf{V} = \bigsqcup_i V_i$, $\textbf{U} = \bigsqcup_i U_i$ both be finite union of disjoint Jordan disks, with $\textbf{U}\subset\textbf{V}$.  Let $G:\textbf{U}\to \textbf{V}$ be a holomorphic proper map, such that all its critical points are contained in the filled-in Julia set $K(G):= \bigcap_{n\geq0}G^{-n}(\textbf{V})$. Denote by $C(G)$ (might be empty) the collection of critical points of $F$. A connected component of $G^{-n}(\textbf{V})$ is called a puzzle piece of depth $n$. A puzzle piece of depth $n$ containing $z$ is denoted by $P_n(z)$.

From now on, we add the assumption that $G$ has only one critical point $c$. Let $z\in K(G)$. The {\it tableau} $T_G(z)$ is the two dimensional array $\{P_{n,l}(z)\}_{n,l\geq0}$ with $P_{n,l}(z) = P_n(G^l(z))$. The position $(n,l)$ is called {\it critical} if $c\in P_{n,l}(z)$. The tableau $T_G(z)$ is called {\it critical} if for all $n\geq 0$, there exists $l_0\geq0$ such that $(n,l)$ is critical. It is called {\it periodic} if there exists $k\geq 1$ such that $P_{n,k}(z) = P_{n,0}(z)$ for all $n\geq 0$; otherwise it is called {\it aperiodic}.  We say that $P_{n+k}(c)$ ($k\geq 1$) is a child of $P_n(c)$ if $G^{k-1}:P_{n+k-1}(G(c))\to P_{n}(c)$ is conformal. All children of $P_n(c)$ are naturally ordered by inclusion of sets. Hence if $P_n(c)$ has at least one child, the {\it first child} of $P_n(c)$ is well-defined.

We say that $x$ combinatorially accumulates to $y$, written as $x\to_c y$, if for any $n\geq0$, there exists $m> n$ such that $P_{m}(F^{m-n}(x)) = P_n(y)$. An aperiodic critical tableau $T_G(c)$ (or a critical point $c$) is called {\it recurrent} if $c\to_c c$.\\

\paragraph{\textbf{Construction of $\textbf{U},\textbf{V},G$.}} 
Now we come back to our specific context of cubic Siegel polynomials. Let $f:=f_b$ satisfy Assumption \ref{assum.well-defined}. Define the dynamical graph by (\ref{eq.dyna.graph}). Define $\textbf{V}$ to be $E^\infty_{<r}\setminus I_0$, $\textbf{U} := f^{-1}(\textbf{V})$, $G := f$. Thus the critical point $1/b$ is always on the graph $I_0$ and the puzzle piece $P_n(b)$ is well-defined for all $n\geq 0$.

\begin{lem}\label{lem.compactly.contain}
   For any $z\in K(f)$ satisfying $f^{n}(z)\not\in\overline{R^\bB_0}$ for all $n\geq0$, there exists $n_z\geq1$ such that $P_{n_z}(z)\Subset P_0(z)$.
\end{lem}
\begin{proof}
   By Corollary \ref{cor.decomposition}, let $W_x$ with $x\in\partial D(f)$ be the wake containing $z$. Thus there exists $n_x\geq0$ such that $[x] = (e^{-2\pi in_x\theta};0)$.  Let $B_x$ be the bubble with root $x$. Again by Corollary \ref{cor.decomposition}, let $W_y$ with $y\in\partial B_x$ be the wake containing $z$. Consider two cases respectively:
    \begin{enumerate}
        \item $n_x\not = 0,1$. Take two joints $z_1,z_2\in\partial B_x$ such that $y$ is contained in the open interval $(z_1,z_2)\subset\partial B_x$ not containing $x$. Notice that there exists $n_z\geq 1$ such that there are bubble rays $R^{\bB}_{\tau_1},R^{\bB}_{\tau_2}\subset I_{n_z}$ containing $z_1,z_2$ respectively. Let $R^\infty_{t_1},R^\infty_{t_2}$ be the external ray colanding with $R^{\bB}_{\tau_1},R^{\bB}_{\tau_2}$ respectively, $S$ be the connected component of $\C\setminus (R^{\bB}_{\tau_1}\cup R^{\bB}_{\tau_2}\cup\overline{R^\infty_{t_1}\cup R^\infty_{t_2}})$ such that $z\in S$. Thus by construction $P_{n_z}(z)\subset S\Subset P_0(z)$. 
        \item $n_x = 0$ or $1$. We only prove for $n_x=0$ and it will the the same for $n_x=1$. By Lemma \ref{lem.loc.alpha}, $z$ does not belong to the infinite nested sequence of wake associated to $R^\bB_0(f)$. Thus by Corollary \ref{cor.decomposition}, there exists a bubble $B\subset R^\bB_0(f)$, a joint $y'\in\partial B$ and a bubble $B'$ attached at $y'$, such that $B'\not\subset R^\bB_0(f)$, and $z\in W_{y'}$. We can then repeat the argument in Case 1.
    \end{enumerate}
\end{proof}

\begin{lem}
    If the tableau $T_{f}(b)$ is periodic, then $f$ is separable.
\end{lem}
\begin{proof}
Let $P_+,P_-$ be the two level zero puzzle pieces with respect to the graph $I_0(f)$ (\ref{eq.dyna.graph}). For a point $z\in J(f)$ such that $f^n(z)\not\in I_0(f)$ for all $n\geq 0$, we define its itinerary with respect to $I_0(f)$ by assigning the binary expansion $\epsilon_0\epsilon_1...$ such that $\epsilon_i = 0$ if $f^i(z)\in P_-$ and $\epsilon_i = 1$ if $f^i(z)\in P_+$. Thus $f$ acts on the itinerary as a left shift. From the definition of periodic tableau, the itinerary of $b$ is periodic under the action of $f$. Suppose the contrary that $f$ is not separable. Then by Proposition \ref{nonsep}, $b$ is contained in some wake impression $\mathrm{Imp}_W(R^\bB)$ where $R^\bB$ is either a strictly pre-periodic bubble ray or an irrational bubble ray. In the first case, $\mathrm{Imp}_W(R^\bB)$ is equal to $\{b\}$, and thus the itineray of $b$ is strictly pre-periodic. In the second case, by Proposition \ref{prop.irrational wake trivial}, the wake arc $\mathrm{Arc}_W(R^\bB)$ is trivial. Hence the itineray of $b$ is the same to that of the accumulation set of $R^\bB$, which is not periodic. Thus in both cases we get a contradiction. This finishes the proof.
\end{proof}

\begin{lem}\label{lem.depth.difference}
    Suppose $T_f(b)$ is recurrent (in particular not periodic). Let $P_{n_j}(b)$ be any sequence of puzzles pieces with $n_j\to+\infty$ as $j\to+\infty$. Let $P_{m_j}(b)$ be the first child of $P_{n_j}(b)$. Then $m_j-n_j\to+\infty$. 
\end{lem}
\begin{proof}
    Suppose the contrary. Up to taking a subsequence, we may assume that $m_j-n_j$ is constant. Notice that $P_{m_j}(f^{m_j-n_j}(b))$ is critical for all $j$, hence $T_f(b)$ is periodic.
\end{proof}

Next we recall the second key ingredient in our proof: the {\it favorite nest}.

\begin{defn}[\cite{KaLy}]\label{def.GPL.map}
    A {\it generalized polynomial-like map}
(GPL map) is a holomorphic map $$g:\bigcup_i W_i\to V,$$ 
where $V \subset \C$ is a topological disk
and $W_i \Subset V$ are topological disks with disjoint closures such that the restrictions
$g: W_i \to V$ are branched coverings, and moreover, all but finitely many of them
have degree one. A GPL map $g$ is called {\it unicritical} if it has a single critical point. 
\end{defn}

By Lemma \ref{lem.compactly.contain}, there exists $L>0$ such that $P_{L}(b)\Subset P_0(b)$. Now suppose $T_f(b)$ is recurrent. Then there exists a sequence of pairs of nested puzzle pieces $P_{n_i+L}(b)\Subset P_{n_i}(b)$ with $n_i\to+\infty$. By Lemma \ref{lem.depth.difference}, we may suppose that $n_0\geq L$ is large enough so that the first child (denoted by $P_{m_0}(b)$) of $P_{n_0}(b)$ satisfies $m_0-n_0\geq L$. Hence $P_{m_0}(b)\Subset P_{n_0}(b)$. Consider the first return map
$$R_{V}:\mathrm{Dom}\,R_V\to V,\,\, V := {P_{m_0}(b)},$$
where $\mathrm{Dom}\,R_V := \{z\in V;\,\exists n\geq 1\text{ such that }f^n(z)\in V\}$ and $R_V(z) := f^k(z)$, where $k\geq 1$ is the smallest integer such that $f^k(z)\in V$. 

Let $z\in\mathrm{Dom}\,R_V$ and $W\subset \mathrm{Dom}\,R_V$ be the connected component containing $z$. 
\begin{itemize}
    \item We first claim that $W\Subset V$. Let $k$ be the first return moment of $z$. Since $P_{m_0}(b)$ is the first child of $P_{n_0}(b)$, we must have $k\geq m_0-n_0$. If we pull back $P_{n_0}(b)$ along the orbit of $z$, we will get $P_{n_0+k}(z)$, which compactly contains $W=P_{m_0+k}(z)$. Notice that $P_{n_0+k}(z)\subset P_{m_0}(b)$ since $n_0+k\geq m_0$. Thus $W\Subset P_{m_0}(b)$.
    
    \item Next we claim that $\overline{W}\cap \overline{W'} = \emptyset$ for any two different components $W,W'$ containing $x$ and $x'$ respectively. Suppose the contrary. Take any $w_0\in\overline{W}\cap \overline{W'}$. Let $k,k'$ be the first return moment of $x,x'$ respectively. Notice that $k\not = k'$, otherwise $\partial W$ and $\partial W'$ will be sent to $\partial V$ by $f^k$, and hence $w_0$ is a critical point of $f^k$, however the graph $I_\infty$ (recall (\ref{eq.dyna.graph})) contains no critical point of $f$. We may thus assume that $k>k'$. Let $z_0 = f^{k'}(w_0)$. Then $z_0,f^{k-k'}(z_0)\in\partial V$. On the other hand, $f^{i}(\partial V)\cap \partial V = \emptyset$ for $0<i<m_0-n_0$, since $V$ is the first child of $P_{n_0}(b)$ and $V\Subset P_{n_0}(b)$. By invariant property of the graph $I_n$, $f^i(\partial V)\cap\partial V = \emptyset$ for all $i\geq m_0-n_0$. Thus $f^{k-k'}(z_0)\not\in\partial V$, a contradiction.
\end{itemize} 
The above reasoning confirms that $R_{V}:\mathrm{Dom}\,R_V\to V$ is a GPL map. For a unicritical GPL map whose critical tableau is recurrent, Kahn and Lyubich construct the {\it Modified principal nest} \cite{KaLy} around the critical point
\begin{equation}\label{eq.modifiedprincipal}
    V=:V^0\Supset W^0\Supset V^1\Supset W^1\Supset...\ni b,
\end{equation}
where $V^i,W^i$ are carefully chosen puzzle pieces containing $b$. They prove the complex a priori bound of this nest and show the local connectivity of the Julia set \cite[Thm. B]{KaLy}. In \cite{AvKaLySh}, the modified principal nest (\ref{eq.modifiedprincipal}) is slightly adjusted to the {\it favorite nest}: each $W^i$ is the first child of $V^i$ and $V^{i+1}$ is a carefully chosen child of $V^i$ (called the favorite child, see \cite[\S 2]{AvKaLySh}). In particular, they are also puzzle pieces containing $b$. For the favorite nest, one also has the a priori bound:

\begin{prop}[{\cite[Prop. 2.5]{AvKaLySh}}]\label{prop.apriori.bound}
    There exists a uniform constant $\mu>0$ such that $\mathrm{mod}(V^i\setminus \overline{W^i})>\mu$.
\end{prop}

We will apply the following theorem later in the proof of the rigidity in the recurrent case \S \ref{subsec.recurrentcase}:

\begin{thm}[{\cite[Thm. 3.1]{AvKaLySh}}]\label{thm.teichmuller}
    Let $f,\Tilde{f}$ satisfy Assumption \ref{assum.well-defined}. Suppose both $T_{{f}}(b),T_{\Tilde{f}}(\Tilde{b})$ are recurrent. Let $(V^i,W^i)$, $(\tilde{V}^i,\Tilde{W}^i)$ be their corresponding favorite nest. Suppose there exists a homeomorphism $h:\C\to\C$ such that $h(V^i) = \Tilde{V}^i$ for $0\leq i\leq m$ and $h\circ f(z) = \Tilde{f}\circ h(z)$ for $z\not\in V^m$. Assume furthermore that
    \begin{enumerate}
        \item[$(\mathrm{1}).$] there exists $\delta>0$ such that $\mathrm{mod}(V^i\setminus \overline{W^i})>\delta$, $\mathrm{mod}(\tilde{V}^i\setminus \overline{\Tilde{W}^i})>\delta$ for $0\leq i\leq m-1$;
        \item[$(\mathrm{2}).$] $h|_{\partial V^0}$ extends to a $K$-qc map $(V^0,b)\to (\Tilde{V}^0,\Tilde{b})$.
    \end{enumerate}
    Then $h|_{\partial V^m}$ extends to a $K'$-qc map $(V^m,b)\to (\Tilde{V}^m,\Tilde{b})$ with $K' = K'(\delta,K)$.
\end{thm}

\section{Combinatorial Rigidity: General cases}\label{sec:combinatorial.rigidity}
Suppose both $f := f_b$, $\Tilde{f} := f_{\tilde{b}}$ satisfy simultaneously either Assumption \ref{assum.well-defined}, or Assumption \ref{assum.well-defined'} below:
\begin{assum}\label{assum.well-defined'}
$b\in \mathcal{S}_{\mathrm{fun}}\cap\partial{\mathcal{C}}(\theta)$ and satisfies 
\begin{itemize}
    \item[(\romannumeral1)] $b$ is non-separable.
    \item [(\romannumeral2)]$f^n(b)\in\overline{{R^\bB_0(f)}}\setminus R^\bB_0(f)$ for all $n\geq0$.
\end{itemize}
\end{assum}

The objects associated to $\Tilde{f}$ are also marked with tilde. Notice that under both assumptions, the graph $I_k$ (\ref{eq.dyna.graph}) is well-defined for all $k\geq 0$.

\begin{defn}\label{defn.combina-sigele-general}
 We say that $f,\Tilde{f}$ are {\it combinatorially equivalent}, if for any $k\geq 0$, there exists a homeomorphism $$\phi_k: I_k\to \Tilde{I}_k$$
(where $\Tilde{I}_k$ stands for the graph (\ref{eq.dyna.graph}) for $\Tilde{f}$), such that $\phi_{k-1}\circ f= \Tilde{f}\circ\phi_k$. 
\end{defn}

The main purpose of this section is to prove the following rigidity theorem:
\begin{thm}\label{thm.rigidity}
    Suppose $f,\Tilde{f}$ satisfy simultaneously either Assumption \ref{assum.well-defined}, or Assumption \ref{assum.well-defined'} . If $f,\Tilde{f}$ are combinatorially equivalent, then $f = \Tilde{f}$.
\end{thm}

Notice that if $f$ satisfies Assumption \ref{assum.well-defined'}, then the free critical point will eventually hit $\alpha$, the co-landing point of $R^\bB_0$ and $R^\infty_0$. In this case, the critical tableau $T_f(b)$ is not well-defined since $b$ belongs to some graph $I_k$. However, the co-critical point is the landing point of some external ray $R^\infty_t$ which is pre-periodic to $R^\infty_0$, hence by Proposition \ref{pre per para ray}, $b$ is the landing point of the parameter external ray $\mathcal{R}^\infty_t$. Hence Theorem \ref{thm.rigidity} holds immediately in this case.

In the rest of the section, we always assume that $f$ and $\Tilde{f}$ satisfy Assumption \ref{assum.well-defined}. Then the critical tableau is well-defined. It is immediate from the definition that if $f,\Tilde{f}$ are combinatorially equivalent, then $T_f(b)$ and $T_{\Tilde{f}}(\Tilde{b})$ are simultaneously recurrent or non-recurrent. 

\subsection{No Lebesgue density point on the non recurrent set}
\text{ }\\
We first recall the following classical inequality for modulus of annulus:
\begin{lem}[\cite{Ly2}]\label{lem.inequality.module}
    Let $U'\Subset U$, $V'\Subset V$ be two pairs of nested topological disks. Let $g:(U',U)\to (V',V)$ be a proper holomorphic mapping. Then \[\mathrm{mod}(U\setminus\overline{U'})\leq\mathrm{mod}(V\setminus\overline{V'})\leq \mathrm{deg}(g)\cdot\mathrm{mod}(U\setminus\overline{U'}).\]
\end{lem}

We will also need the notions of shape and turning to control the geometry of puzzle pieces:
\begin{defn}[Shape and turning \cite{KoShvS}]
    The {\it shape} $\mathbf{S}(U,z)$ of a topological disk $U\in\C$ at $z\in U$ is defined to be $$\sup_{w\in\partial U}|z-w|/\inf_{w\in\partial U}|z-w|.$$
    The {\it turning} $\mathbf{T}(K;z_1,z_2)$ of a compact set $K\subset\C$ about $z_1,z_2\in K$ is defined to be $\mathrm{diam}(K)/|z_1-z_2|$.
\end{defn}
\begin{lem}[{\cite[Lem. 6.1]{QiWaYi}}]\label{lem.shape-turning}
   Let $U'\Subset U$, $V'\Subset V$ be topological disks, $R:U\to V$ be a proper map of degree $d$, $U'$ is a component of $R^{-1}(V')$.  Suppose $\mathrm{mod}(V\setminus V') \geq M > 0$. Then,
\begin{itemize}
    \item[$(\mathrm{\romannumeral1})$.] there is a constant $C(d,M)$ such that for all $z\in U'$, 
    $$\mathbf{S}(U',z)\leq C(d,M)\cdot\mathbf{S}(V',R(z)).$$
     \item[$(\mathrm{\romannumeral2})$.]There is a constant $D(d,M) > 0$ such that for any connected and compact subset $K$ of $U'$ with $\#K \geq 2$ and any $z_1, z_2 \in K$,
$$\mathbf{T}(K; z_1, z_2) \leq D(d, M)\cdot\mathbf{T}(R(K); R(z_1), R(z_2)).$$
\end{itemize}
\end{lem}
The following lemma is a useful criterion for a point {\it not} to be a Lebesgue density point in the Julia set.
\begin{lem}[{\cite[Lem. 5.6]{Wa}}]\label{lem.shape.distortion}
    Let $R$ be a rational map, $z\in J(R)$. Suppose that there exist integers $D_z>0$, $0<n_1<n_2<...$, a sequence of disk neighborhoods $U_j'\Subset U_j$ of $z$, two disks $V_z'\Subset V_z$, such that $R^{n_j}:U_j\to V_z$, $R^{n_j}:U_j'\to V_z'$ are proper maps of degree less than $D_z$. Then 
 \begin{itemize}
     \item[$(\mathrm{\romannumeral1})$.] $\mathrm{diam}(U_j')\to 0$;
     \item[$(\mathrm{\romannumeral2})$.] $z$ is not a Lebesgue density point of $J(R)$.
 \end{itemize}
\end{lem}

Now we are ready to prove the main lemma of this section:
\begin{lem}\label{lem.zero.measure}
    Suppose that $f$ satisfy Assumption \ref{assum.well-defined}. For all $N\geq 0$, $J(f)$ is locally connected at $\{z\in J(f);\forall l\geq 1, f^l(z)\not\in P_N(b)\}$. Moreover this set has zero Lebesgue measure.
\end{lem}
\begin{proof}
   Let us take any $z$ in the prescribed set. consider three cases:
\paragraph{(1).} $\{f^n(z)\}$ accumulates to $z_0\in J(f)$ such that $f^n(z_0)\not\in\overline{R^\bB_0}$ for all $n$.

By Lemma \ref{lem.compactly.contain}, there exists $n_{z_0}$ such that $P_{n_{z_0}}(z_0)\Subset P_{0}(z_0)$. Since $z$ accumulates to $z_0$, there exists a subsequence $f^{l_i}(b)\in P_{n_{z_0}}(z_0)$. Pull back $P_{0}(z_0)\setminus\overline{P_{n_{z_0}}(z_0)}$ along the orbit of $z$, we get a sequence of non degenerate annuli $P_{l_i}(z)\setminus\overline{P_{n_{z_0}+l_i}(z)}$. Since $z\not\to_c b$, the degree of $f^{l_i}:P_{l_i}(z)\to P_0(z_0)$ is uniformly bounded. Thus by Lemma \ref{lem.shape.distortion} (\romannumeral1) (\romannumeral2), $\bigcap_n \overline{P_n(z)} = \{z\}$ and $z$ is not a Lebesgue density point.

\paragraph{(2).} $\{f^n(z)\}_{n\geq0}$ accumulates to $\alpha$, the landing point of $R^\bB_0$.

Consider another graph defined by
$$I'_0 = R^\bB_{1/3}\cup R^\bB_{2/3}\cup\overline{\mathbf{R}^{'\infty}}\cup E^\infty_r,\,\, I'_n = f^{-n}(I'_0),$$
where $\mathbf{R}^{'\infty}$ is the union of external rays that coland with $R^\bB_{1/3}$ or $R^\bB_{2/3}$ (notice that $R^\bB_{1/3},R^\bB_{2/3}$ are $2$-periodic bubble rays). Thus $\alpha\not\in I'_0$ and is contained in the $0$-depth puzzle piece $P'_0(\alpha)$ defined by $I'_0$. Applying the same proof of Lemma \ref{lem.compactly.contain}, it is easy to show that there exists $L>0$ such that $P'_L(\alpha)\Subset P'_0(\alpha)$. Then using the same argument as Case $(\mathbf{1})$, we get the result.
\paragraph{(3).} $\{f^n(z)\}_{n\geq0}$ accumulates at $1/b$.\\
We will work with the associated Blaschke product $F$. Recall that $J_1(F) = \varphi(J(f))$. Suppose that $w =  \varphi(z)\in J_1(F)$ accumulates to $1$, where $\varphi$ is given by Proposition \ref{prop.surgery}. Consider the first hit map of $w$ to $C^n$: $$F^{m_n}: C^{n,m_n}\to C^{n,0} := C^{n}.$$
For all $m\geq 0$, denote by $C^{n-m,k}$ the $k$-th preimage of $C^{n-m}$ along the orbit of $z$. By Lemma \ref{lem.bounded.first.hit}, for $n$ large enough, $F^{m_n}:C^{n-7,m_n}\to C^{n-7}$ has bounded degree. By Lemma \ref{lem.Cn-Dn}, $C^{n-7}\setminus\overline{C^n}$ has modulus bounded from below; thus by Lemma \ref{lem.inequality.module}, $C^{n-7,m_n}\setminus\overline{C^{n,m_n}}$ has modulus bounded from below. On the other hand, we claim that $C^{n+7,m_{n+14}}\subset C^{n,m_n}$. Indeed, otherwise there exists $w_0\in\partial C^{n,m_n}\cap C^{n+7,m_{n+14}}$. Thus for all $n'\geq m_n+q_{n+1}$, $F^{n'}(z_0)\in I_0(F)\cup\kappa(I_0(F)$; while $F^{m_{n+14}}(z_0)\in C^{n+7}$, a contradiction. Thus the claim is proved and $C^{n,m_n}$ shrinks to $w$ by Grötzsch inequality. The local connectivity of $J_1(F)$ at $w$ follows.

Finally we prove that $w$ is not a Lebesgue density point of $J_1(F)$. By Corollary \ref{cor.commensurable}, for each $n$, there exists a round disk $V^{n}\subset (D'\cap C^{n})$ of commensurable size with $C^{n}$. In particular $V^n\cap J_1(F) = \emptyset$. Let $V^{n,m_n}$ be a connected component of $F^{-m_{n}}(V_n)$ in ${C}^{n,m_{n}}$. Let $w_n\in V^{n,m_n}$ be a pre-image of the center of $V^{n}$ under $F^{m_{n}}$. For each $n$, take a compact topological disk $K^n\subset{C}^{n,m_{n}}$ such that $V_{n,m_n}\subset K^n$ and $K^n$ has commensurable size with ${C}^{n,m_{n}}$. Since $F^{m_n}:C^{n-7,m_n}\to C^{n-7}$ has bounded degree by Lemma \ref{lem.bounded.first.hit-Cf} and $C^{n-7}\setminus\overline{C^n}$ has modulus bounded from below, the turning $\mathbf{T}(K^n,w^n_1,w^n)$ is bounded by Lemma \ref{lem.shape-turning} (\romannumeral2), where $w^n_1\in \overline{V^{n,m_n}}$ is the furthest point to $w^n$. Hence ${C}^{n,m_{n}}$ and $V^{n,m_n}$ have commensurable size. Notice that the shape $\mathbf{S}(V^{n,m_n},w_n)$ is bounded by Lemma \ref{lem.shape-turning} (\romannumeral1). Hence ${C}^{n,m_{n}}$ contains a round disk with commensurable size that does not intersect with $J_1(F)$. Thus $w$ is not a Lebesgue density point in $J_1(F)$.
\end{proof}

\begin{rem}\label{rem.zero.measure.nonrecur}
    If $T_f(b)$ is non recurrent, then using the same proof as Case (\textbf{1}) in Lemma \ref{lem.zero.measure}, it is easy to show that $J(f)$ is locally connected at 
    $$\{z\in J(f);\,\forall n\geq 0,\,\exists l\geq 1,\,f^l(z)\in P_n(b)\}$$
    and that this set has zero Lebesgue measure.
\end{rem}
In particular, we get
\begin{cor}\label{cor.zero.measure.nonrecur}
     If $T_f(b)$ is non recurrent, then $J(f)$ is locally connected and has zero Lebesgue measure.
\end{cor}

\subsection{Recurrent case}\label{subsec.recurrentcase}
\begin{proof}[{\it Proof of Theorem \ref{thm.rigidity}: {recurrent case.}}]Suppose that $T_f(b),T_{\tilde{f}}(\Tilde{b})$ are recurrent. We give a proof following the idea in \cite{AvKaLySh} (see also \cite[\S5.3]{Wa}). 
    \vspace{0.3cm}
 \paragraph{{\it Step 1: constructing weak pseudo-conjugacy $h_n$.}} 
 
 Since $I_0(f)$ moves holomorphically (Lemma \ref{lem.holo.motion}), there is a homeomorphism $h_0:I_0\cup\overline{E^\infty_{>r}}\to \Tilde{I}_0\cup\overline{\Tilde{E}^\infty_{>r}}$ induced by the coordinates of linearization of siegel disks and the Böttcher coodinates at infinity (recall the notation $E^\infty_{>r}$ in (\ref{eq.equipotential>r})). In particular $h_0$ 
 preserves dynamics and is conformal in the interior of its domain of definition. By Slodkowski extension theorem, $h_0$ can be extended to a qc map $h_0:\C\to\C$. Since $f,\Tilde{f}$ are combinatorially equivalent, one can modify quasiconformally $h_0$ on $P_0(v)$, where $v$ is the critical value of $f$, so that $h_0(v) = \Tilde{v}$. Again, since $f,\Tilde{f}$ are combinatorially equivalent, it is easy to check by using the same argument as Theorem \ref{thm.rigid.Zak}, that $h_0$ admits a $(f,\Tilde{f})$-lift to $h_1$, normalised by $h_1(0) = 0$. Now we modify $h_1$ near $v$ so that $h_1(v) = \Tilde{v}$, and then we can lift by $(f,\Tilde{f})$ it to $h_2$ since again, $f$ and $\Tilde{f}$ are combinatorially equivalent. Repeat this procedure, we get a sequence of qc maps $h_n$. However, up to now, we have no control on their dilatations.

   \vspace{0.3cm}
 \paragraph{{\it Step 2: Bounding dilatation by the critical piece.}} 
 
 From Step 1, we see that $h_n$ conjugates $f$ to $\Tilde{f}$ on $I_n(f)$. In this step, we aim to prove the following statement: if $h_n|_{\partial P_n(b)}$ admits a $K$-qc extension $G_n: P_n(b)\to \Tilde{P}_n(\Tilde{b})$, then it has a further $K$-qc extension $G_n:\C\to\C$ conjugating $f$ to $\Tilde{f}$ on $\C\setminus \overline{P_n(b)}$.

By hypothesis, we can modify $h_n$ on $P_n(b)$ so that it is $K$-qc on $P_n(b)$. By Rickman's lemma (cf. \cite[Lem. 2]{DoHu2}), we may assume that the modified map, denoted by $H_{0}$, is also qc on $\C$. Let $v,\Tilde{v}$ be the critical value of $b,\Tilde{b}$ respectively. Since $f,\Tilde{f}$ are combinatorially equivalent, ${H}_{0}:\C\setminus \overline{P_{n-1}(v)}\to \C\setminus \overline{\Tilde{P}_{n-1}(\Tilde{v})}$ admits a $(f,\Tilde{f})$-lift to $$H_1:\C\setminus\overline{{P}_{n}({b})\cup P_n(cp)}\to\C\setminus\overline{\Tilde{P}_{n}({\Tilde{b}})\cup \Tilde{P}_n(\Tilde{cp})},$$ where $cp$ (resp. $\Tilde{cp}$) is the co-critical point to $b$ (resp. $\Tilde{b}$). Since $f:P_n(cp)\to P_{n-1}(v)$ is conformal, $H_0:{P_{n-1}(v)}\to {\Tilde{P}_{n-1}(\Tilde{v})}$ can be lifted to $H_1:{P_{n}(cp)}\to{\Tilde{P}_{n}(\Tilde{cp})}$. Again by Rickman's lemma, $H_1:\C\setminus \overline{P_n(b)}\to \C\setminus\overline{\Tilde{P}_n(\Tilde{b})}$ is qc with the same dilatation as $H_0$. By construction, we have $H_1|_{\partial P_n(b)} = H_0|_{\partial P_n(b)} = h_n|_{\partial P_n(b)}$.

Repeat this procedure, we get a sequence of qc maps $\{H_m\}_m$ with bounded dilatation. Hence it admits a limit qc map $G_n$. By construction, $G_n$ conjugates $f$ to $\Tilde{f}$ on $\C\setminus\overline{P_n(b)}$ and is conformal on $F(f)$, $K$-qc on the set $$\{z\in\C;\exists l\geq 1\text{ such that }f^l(z)\in P_n(b)\}.$$
Notice that $H_m$ extends continuously to $\partial P_n(b)$ and is equal to $H_0|_{\partial P_n(b)}$ for all $m$. Thus $G_n$ extends continuously to $\partial P_n(b)$ and coincides with $H_0|_{\partial P_n(b)}$. Hence $G_n$ admits a $K$-qc extension to $P_n(b)$. By Lemma \ref{lem.zero.measure}, $G_n$ is $K$-qc on $\C$.

  \vspace{0.3cm}
 \paragraph{\it Step 3: uniform $K$-qc extension to the critical piece.} Consider the favorite nest $(V^i,W^i)$ (resp. $(\Tilde{V}^i,\Tilde{W}^i)$) for $f$ (resp. $\Tilde{f}$). Let $q_i$ be the depth of $V^i,\Tilde{V}^i$. The construction of $h_n$ in Step 1 implies in particular that $h_{q_i}|_{\partial V^i}$ admits a $K_{q_i}$-qc extension to $V^i$. Hence we can apply Step 2 to get $G_{q_i}:\C\to\C$ which is also $K_{q_i}$-qc and satisfies $G_{q_i}\circ f = \Tilde{f}\circ G_{q_i}$ on $\C\setminus \overline{V^i}$. In particular, we have $G_{q_i}|_{\partial V^0} = h_{q_i}|_{\partial V^0}= h_{q_0}|_{\partial V^0}$. Thus $G_{q_i}$ admits a $K_{q_0}$-qc extension to $V^0$ and assumption (2) in Theorem \ref{thm.teichmuller} is satisfied. Moreover by the complex bound (Proposition \ref{prop.apriori.bound}), the assumption (1) in Theorem \ref{thm.teichmuller} is satisfied. Thus this theorem is applicable to $G_{q_i}$ so that $G_{q_i}|_{\partial V^i} = h_{q_i}|_{\partial V^i}$ admits a $K'$-qc extension to $V^i$ with $K'$ not depending on $i$. Again by Step 2, $h_{q_i}$ admits a $K'$-qc extension $L_{q_i}:\C\to\C$ such that $L_{q_i}\circ f = \Tilde{f}\circ L_{q_i}$ on $\C\setminus\overline{V^i}$. Thus $L_{q_i}$ admits a subsequence converging to some $K'$-qc map $L$ conjugating $f$ to $\Tilde{f}$ on $\C\setminus \overline{P_{n}(b)}$ for all $n$. Proposition \ref{prop.apriori.bound} implies that $\overline{P_n(b)}$ shrinks to $b$, thus $L \circ f = \Tilde{f}\circ L$ on $\C$. By Proposition \ref{prop.quasiconformal.rigid}, $f = \Tilde{f}$. This finishes the proof.
\end{proof}

\subsection{Puzzle disks for polynomials}
\text{ }\\
In order to deal with the non recurrent case, we introduce puzzle disk for polynomials. Roughly speaking, it is the combination of the modified puzzle disk and the round disk.

Let $f$ satisfy Assumption \ref{assum.well-defined}. Recall that by Proposition \ref{prop.surgery}, $f$ is quasiconformally conjugate to a quasi-regular map $P$ on $\C$ by $\varphi$; $P$ is quasiconformally conjugate to the rigid rotation $R_\theta$ on $\D$ by $h$, $P = F$ on $\C\setminus\D$ ($F$ is the associated Blaschke product); moreover $h$ extends to $\partial\D$ so that $h|_{\partial\D}$ is quasi-symmetric. We can extend $h$ to $\C\setminus\D$ by the reflection $\kappa:z\mapsto1/\overline{z}$, so that we obtain a quasiconformal map $h:\C\to\C$.

Recall that $n_0$ is the initial level of the puzzle disk $D^{n_0}$. For $n\geq n_0$, let $\gamma_n\subset\overline{\D}$ be the standard circular arc such that $e^{-2\pi i(q_n+1)\theta},e^{-2\pi i(q_{n+1}+1)\theta}\in\gamma_n$, and the center of the circle containing $\gamma_n$ is the midpoint of the interval $(e^{-2\pi i(q_n+1)\theta},e^{-2\pi i(q_{n+1}+1)\theta})\subset\partial\D$. Notice that $h$ sends $R_\theta^i(1)$ to $c_i=g^i(1)$ for $i\in \mathbb{Z}$, where $g = P|_{\partial\D}$. Recall the interval 
$Jv_n = (-q_{n+1}+1,-q_n+1)_c$ and the topological disk $S^n$ defined in Definition \ref{def.modified-disks}. Notice that $S^n\cap\partial\D = Jv_n$. Let $S^n_P$ be the Jordan domain bounded by $[\partial S^n\cap(\C\setminus\overline{\D})]\cup h^{-1}(\gamma_n)$. See Figure \ref{fig.SnP}. Let $O^n_P$ be the Jordan domain bounded by $h^{-1}(\gamma_n)\cup\kappa(h^{-1}(\gamma_n))$. Define the puzzle disk for $f$ of depth $n$ by 
$$C^n_f := \text{ the connected component of }f^{-1}(\varphi^{-1}(S_P^n)) \text{ that intersects }\partial\D.$$

\begin{figure}
\centering 
\includegraphics[width=0.5\textwidth]{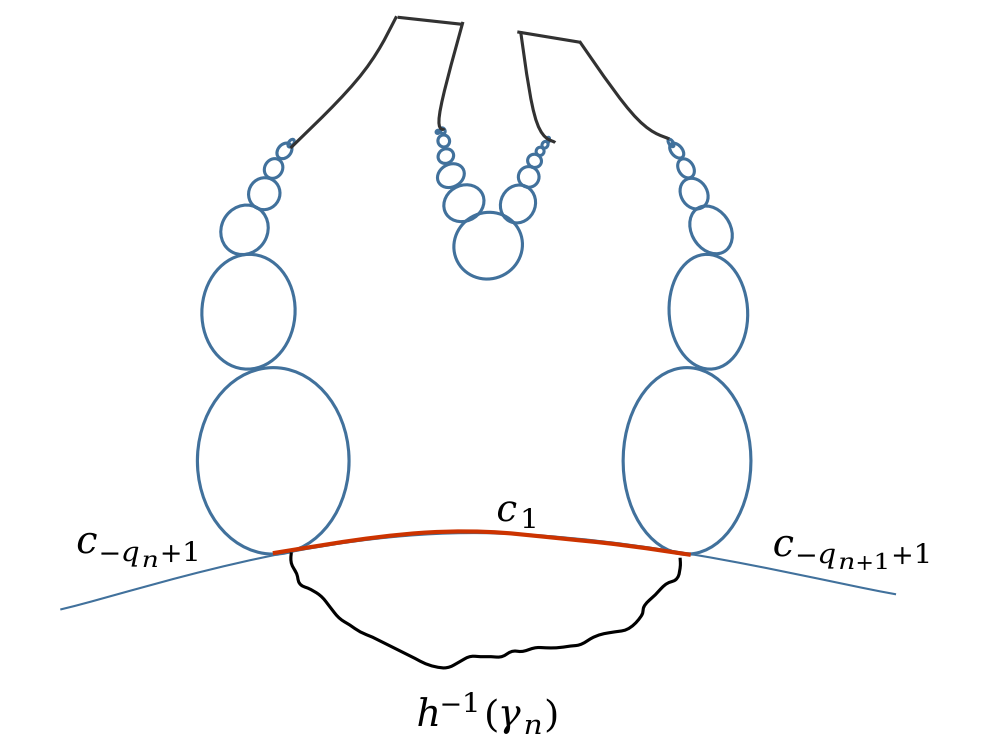} 
\caption{A schematic picture of $S^n_P$} 
\label{fig.SnP} 
\end{figure}

\begin{lem}\label{lem.modulus.S_f}
    There exists $\Tilde{N}$ and $\tilde{M}$ such that for all $n\geq n_0$, ${1}/{\tilde{M}
    }<\mathrm{mod}(C^{n}_f\setminus\overline{C^{n+\Tilde{N}}_f})<\Tilde{M}$.
\end{lem}
\begin{proof}
    Notice that $P\circ\varphi(C^n_f) = S^{n}_P$. Since $\varphi$ is quasiconformal, $P:\varphi(C^n_f)\to S^n_P$ is a quasi-regular degree two covering map, it suffices to prove the estimate $${1}/{\tilde{M}
    }<\mathrm{mod}(S^{n}_P\setminus\overline{S^{n+\Tilde{N}}_P})<\Tilde{M}.$$ 
    
    By Lemma \ref{lem.Cn-Dn}, $S^n\setminus \overline{S^{n-7}}$ has modulus bounded from below. By Lemma \ref{lem.commensurable-round-disk}, there exists $N'$ such that $$S^{n+2N'}\subset B^{n+N'}(c_1)\subset S^{n}.$$ 
    On the other hand, since $\theta$ has bounded type, the modulus of the annulus bounded by $\gamma_n\cup\kappa(\gamma_n)\cup\gamma_{n-2}\cup\kappa(\gamma_{n-2})$ is bounded from below. Since $h:\C\to\C$ is quasi-conformal, $O^n_P\setminus \overline{O^{n-2}_P}$ also has bounded modulus.
    Again by Lemma \ref{lem.commensurable-round-disk}, we have
    $$O^{n+2N'}_P\subset B^{n+N'}(c_1)\subset O^{n}_P$$
    (modify $N'$ larger if necessary). From the construction of $S^n_P$, we have $S^n_P\cap(\C\setminus\D) = S^n\cap(\C\setminus\D)$ and $S^n_P\cap\D = O^n_P\cap\D$. Thus we get $S^{n+2N'}_P\subset B^{n+N'}(c_1)\subset S^{n}_P$. Thus 
$$(B^{n+N'}(c_1)\setminus\overline{B^{n+3N'}(c_1)})\subset (S^{n}_P\setminus \overline{S^{n+4N'}_P})\subset(B^{n-N'}(c_1)\setminus\overline{B^{n+5N'}(c_1)}).$$
   If we take $\Tilde{N} = 4N'$, then the desired estimate follows.
\end{proof}
\begin{lem}\label{lem.bounded.shape-Cf}
    There exists $N_1\geq \Tilde{N}$, $M_1$ and $\nu>0$, such that for all $n\geq n_0+N_1$ and all $z\in C^{n}_f$, $\mathbf{S}(C^{n-N_1}_f,z)\leq M_1$.
    
\end{lem}
\begin{proof}
    By Lemma \ref{lem.modulus.S_f}, and Theorem \ref{thm.mcmullen}, there exists $N_1\geq \Tilde{N}$, $M_0>0$, such that $C^{n-N_1}_f\setminus \overline{C^{n}_f}$ contains a round annulus $A^n$ centered at $z=1/b$, such that ${1}/{M_0}<\mathrm{mod}(A^n)<M_0$. Let $T_+^n,T^n_-$ be the outer/inner component of $\partial A^n$. Then $T_+^n$, $T^n_-$ and $\mathrm{dist}(T^n_+,T^n_-)$ have commensurable size. Notice that the modulus of the annulus bounded by $T^{n-N_1}_-$ and $T^{n}_+$ is also bounded from above, since this annulus is contained in $C^{n-2N_1}_f\setminus \overline{C^{n}_f}$. Thus $T^{n-N_1}_-$ and $T^{n}_+$ have commensurable size. Thus there exists $M_1$ such that
    \begin{equation}
        \mathbf{S}(C_f^{n-N_1},z) = \frac{\sup_{y\in\partial C_f^{n-N_1}}|z-y|}{\inf_{y\in\partial C_f^{n-N_1}}|z-y|} \leq \frac{\mathrm{diam}(T^{n-N_1}_-)}{\mathrm{dist}(T^{n}_+,T^n_-)}\leq M_1.
    \end{equation}
\end{proof}

Using the same argument as in the proof for Lemma \ref{lem.bounded.degree.alongD} and \ref{lem.bounded.first.hit}, we can show that
\begin{lem}\label{lem.bounded.first.hit-Cf}
Suppose that $z\in J(f)$, the orbit of $z$ accumulates to $\partial\D$, and that there exists $N\geq0$ such that for all $l\geq 1$, $f^l(z)\not\in P_N(b)$. For any $m\geq 0$, denote by $C^{n-m,k}_f$ the pull back of $C^{n-m}_f$ along the orbit of $z$. Consider the first hit map of $z$ to $C^n_f$: $f^{m_n}: C^{n,m_n,k}_f\to C^n_f$. Then there exists uniform constants $M(m,\theta,N)$, $d(m,\theta)$, such that for $n\geq M(m,\theta,N)$, 
$$f^{m_n}:{C}^{n-m,m_n}_f\to{C}^{n-m}_f$$
has bounded degree $d(m,\theta)$.
\end{lem}

\subsection{Non recurrent case}
\text{ }\\
Next we prove the rigidity theorem for non recurrent case.
\begin{proof}[{\it Proof of Theorem \ref{thm.rigidity}: {non recurrent case.}}] Recall Step 1 and Step 2 in the proof of the recurrent case in \S\ref{subsec.recurrentcase}. Notice that these two steps works without any change for the non recurrent case and we will use the same notations. It suffices to prove that there exists a constant $K>0$, such that $h_{n}|_{\partial P_{n}(b)}$ extends to a $K$-qc map to $P_{n}(b)$. To do this, we need the following criterion of quasiconformal extension (see \cite[Lem. 12.1]{KoShvS}):
\begin{lem}\label{lem.xiaoguang}
  Let $H:\overline{\Omega}\longrightarrow\overline{\Tilde{\Omega}}$ be a homeomorphism between two closed Jordan domains, $k\in(0,1)$ be a constant. Let $X$ be a subset of $\Omega$ such that both $X$ and $H(X)$ have zero Lebesgue measure. Assume:
    \begin{enumerate}
        \item[$({\mathrm{a}}).$] $|\overline{\partial}H|\leq k|\partial H|$ a.e. on $\Omega\setminus X$.
        \item[$({\mathrm{b}}).$] There is a constant $M>0$ such that for all $x\in X$, there is a sequence of open topological disks $\cdot\cdot\cdot D_2\Subset D_1$ containing $x$ such that $\bigcap_j \overline{D_j} = \{x\}$ and $\sup_j\mathbf{S}(D_j,x)\leq M$, $\sup_j\mathbf{S}(H(D_j),H(x))<\infty$.
    \end{enumerate}
Then $H|_{\partial\Omega}$ extends to a $K$-qc map on $\Omega$, where $K = K(k,M)$. 
\end{lem}
\begin{rem}
    In \cite[Lem. 12.1]{KoShvS}), the lemma was stated for a sequence of round disks. In the lemma above, we relax the condition to bounded shape topological disks and the same proof goes through without change. 
\end{rem}

By Lemma \ref{lem.zero.measure} and Remark \ref{cor.zero.measure.nonrecur}, $J(f),J(\Tilde{f})$ are locally connected. Thus we can construct a topological conjugacy $\psi$ between $f,\Tilde{f}$ as follows: define $\psi$ on $D(f)$ by $(\Tilde{\phi})^{-1}\circ\phi:D(f)\to{D}(\Tilde{f})$, where $\phi,\Tilde{\phi}$ are linearization coordinates of $f,\Tilde{f}$ such that $(\Tilde{\phi})^{-1}\circ\phi(1/b) = 1/\Tilde{b}$. Pull-back $\psi$ by $f,\Tilde{f}$ to spread it to the interior of $K(f)$; define $\psi$ on the super-attracting basin of infinity of $f$ to be $\Tilde{\phi}_{\infty}^{-1}\circ \phi_\infty$, where $\phi_\infty,\Tilde{\phi}_{\infty}$ are  Böttcher coordinates of $f,\Tilde{f}$. These two conjugacies extends continuously to the Julia set $J(f)$ and coincide there since $J(f)$ is locally connected and $f,\Tilde{f}$ are combinatorially equivalent. The global conjugacy $\psi$ is defined to be the sewing of these two extended maps. 

In order to apply Lemma \ref{lem.xiaoguang}, we set $$\Omega = P_n(b),\,\, X = J(f)\cap P_n(b),\,\, H = \psi.$$

By Lemma \ref{lem.zero.measure} and Remark \ref{cor.zero.measure.nonrecur}, $J(f)$ has zero measure. By construction, $\psi$ is conformal on $\C\setminus J(f)$. It remains to verify condition ($\mathrm{b}$) in Lemma \ref{lem.xiaoguang}. Let us take any $z\in J(f)$. We consider three cases:
\begin{itemize}
     \item $z$ accumulates to some $z_0\in J(f)$ such that $f^n(z)\not\in\overline{R^\mathbf{B}_0}$ for all $n$. 
     
     Suppose $f^{l_j}(z)$ converges to $z_0$. Pull back $P_0(z_0)$ along the orbit of $z$, we get a sequence of puzzle pieces $P_{l_i}(z)$. Take a small round disk $B\Subset P_0(z_0)$ containing $z_0$, such that $\mathrm{mod}(P_0(z_0)\setminus\overline{B})\geq 1$. By the non recurrent hypothesis, $f^{l_i}$ sends $P_{l_i}(z)$ to $P_0(z_0)$ within bounded degree (not depending on $i$ and $w$). Pull back $B$ along the orbit of $z$ we get a sequence of topological disks $B_{l_i}\Subset P_{l_i}(z)$ with $z\in B_{l_i}$. Thus by Lemma \ref{lem.shape-turning} (\romannumeral1), $\mathbf{S}(z,B_{l_i})$ is bounded. It is clear that $B_{l_i}$ shrinks to $z$, since $P_{l_i}(z)$ does. It remains to verify that 
     $\sup_i\mathbf{S}(\psi(B_{l_i}),\psi(z))<\infty$. To see this, we notice that $\psi$ preserves puzzle pieces. Hence $\psi(P_{l_i}(z)) = \Tilde{P}_{l_i}(\psi(z))$, $\psi({P}_{0}(z_0)) = \Tilde{P}_{0}(\psi(z_0))$. Thus $ \Tilde{f}^{l_i}: \Tilde{P}_{l_i}(\psi(z))\to \Tilde{P}_0(\psi(z_0))$ has bounded degree. The result then follows by Lemma \ref{lem.shape-turning} (\romannumeral1).

     \item $z$ accumulates to the end point of $\overline{R^\bB_0}$ or lands at the end point of $\overline{R^\mathbf{B}_0}$. 
     
     Recall ${P}'_n(\alpha)$ constructed in Case ($\mathbf{2}$) of Lemma \ref{lem.zero.measure}, which is a topological disk containing $\alpha$. We can fix $n = N$ and pull back ${P}'_N(\alpha)$ along the orbit of $z$ and use the same argument as the above case to conclude. 

     \item $z$ lands at $D(f)$ or its orbit accumulates at $1/b$.

     We may suppose that $z$ does not accumulate to the free critical point $b$, otherwise it is covered by the first case.
     
     Let $N_1$ be given in Lemma \ref{lem.bounded.shape-Cf}. Consider the first hit map of $z$ to $C^n_f$:
     $$f^{m_n}:{C^{n,m_n}_f}\to C^n_f.$$
     By Lemma \ref{lem.bounded.first.hit-Cf},
     $f^{m_n}|_{{C}^{n-2N_1,m_n}_f}$ has bounded degree $d(2N_1,\theta)$. By Lemma \ref{lem.modulus.S_f},
     $$\mathrm{mod}(C^{n-2N_1}_f\setminus \overline{C^{n-N_1}_f})\geq M_0.$$
     By Lemma \ref{lem.bounded.shape-Cf}, 
     $$\mathbf{S}(C^{n-N_1}_f,f^{m_n}(z))\leq M_1.$$ 
     By Lemma \ref{lem.shape-turning} (\romannumeral1), $\mathbf{S}({C}^{n-N_1,m_n}_f,z)$ is uniformly bounded by $M_2 = M_2(2N_1,\theta,M_0,M_1)$ (not depending on $n$ and $z$). From the construction of $C^{n}_f$, it is easy to see that $\psi(C^{n-m,k}_f) = C^{n-m,k}_{\Tilde{f}}$ for all $m$ and $k$. Thus   $\mathbf{S}(\psi({C}_{f}^{n-N_1,m_n}),\psi(z)) = \mathbf{S}({C}_{\Tilde{f}}^{n-N_1,m_n},\psi(z))$. The later is uniformly bounded in $n$ for the same reason as $\mathbf{S}({C}^{n-N_1,m_n}_f,z)$.
\end{itemize}
Thus Lemma \ref{lem.xiaoguang} is applicable to $\psi$, and hence $\psi|_{\partial P_n(b)}$ has a $K$-qc extension to $P_n(b)$. While from construction, $\psi|_{\partial P_n(b)} = h_n|_{\partial P_n(b)}$. This finishes the proof.
\end{proof}

\section{Rigidity in the Parameter Plane}\label{sec:parameterplane}

For $n\geq 0$, let
$$
S_n := \{s^n_i := i/2^{n+1} \; | \; -2^n< i < 2^n\}.
$$
If $s \in S_n$, then the parabubble ray $\cR^\bB_s$ is contained in the fundamental region $\cS_{\fun}$ (see \eqref{eq.parabubray} and Definition \ref{def.fundamentalregion}). By \propref{pre per para ray}, $\cR^\bB_s$ lands at a Misiurewicz parameter $b_s$. Moreover, there exists a unique angle $t(s) \in \bbR/\bbZ$ such that the parameter ray $\cR^\infty_{t(s)}$ colands at $b_s$.

Fix $r> 0$. For $n\geq 0$, the {\it parameter graph} $\mathcal{I}_n$ is defined as:
\begin{equation}
\I_n:=\left(\bigcup_{s\in S_n}(\mathcal{R}^{\mathbf{B}}(s)\cup \{b_s\}\cup\mathcal{R}^\infty_{t(s)})\right)\cup\mathcal{E}^\infty_{r/3^n}.
\end{equation}
For $-2^n\leq i < 2^n$, the closure $\cP^n_i$ of a connected component of $\cS_{\fun} \setminus \I_n$ bounded between $\cR^{\mathbf{B}}_{s^n_i}$ and $\cR^{\mathbf{B}}_{s^n_{i+1}}$ is referred to as a {\it parameter puzzle piece of depth $n$}.

\begin{prop}\label{tree corr 1}
Let $b$ be a parameter contained in a parameter puzzle piece of depth $n$. Then $\bR_{s^n_i}(f_b) \subset T(f_b)$ for $-2^{n-1}\leq i < 2^{n-1}$.
\end{prop}

\begin{proof}
We have $\bR_{s^n_i}(f_b) \subset T(f_b)$ unless $b$ (and therefore, also $co_b$) is contained in an iterated image of $\bR_{s^n_i}(f_b)$. However, this can only happen if $b \in \I_{n-1}$.
\end{proof}

Now let $f_b$ satisfy Assumption \ref{assum.well-defined}. Recall the definition of dynamical graph $I_n(f_b)$ of depth $n \geq 0$ given in \eqref{eq.dyna.graph}. By Proposition \ref{tree corr 1}, we have $I_n(f_b) \subset T(f_b)$. For $-2^n\leq i < 2^n$, denote the puzzle piece of depth $n$ for $f_b$ bounded between $\bR_{s^n_i}(f_b)$ and $\bR_{s^n_{i+1}}(f_b)$ by $P^n_i(f_b)$.

\begin{prop}\label{match puzzle}
Let $b$ satisfy Assumption \ref{assum.well-defined}.
We have $b \in \cP^n_i$ for some $n \geq 0$ and $-2^n\leq i < 2^n$ if and only if $co_b \in P^n_i(f_b)$. Moreover, for any $b,\tilde{b}\in\cP^n_i$ satisfying Assumption \ref{assum.well-defined} and any $k\leq n$, there exists a homeomorphism $H_k: I_k(f_b)\longrightarrow I_k(\tilde{f}_b)$, such that $H_{k-1}\circ f_b = \tilde{f}_b\circ H_k$.
\end{prop}

\begin{proof}
The result follows immediately from \lemref{lem.holo.motion}.
\end{proof}

\begin{proof}[Proof of Theorem \ref{thm.main.siegel}]
Consider an infinite nested sequence $\{\cP^n_{i_n}\}_{n=0}^\infty$ of parameter puzzle pieces. Let
$$
\cL := \left(\bigcap_{n=0}^\infty \overline{\cP^n_{i_n}}\right)\setminus\left(\bigcup_{(s,t)\text{ renor} }\mathcal{W}_{s,t}\right),
$$
where the union of $\mathcal{W}_{s,t}$ takes over all renormalizable parameter wakes (recall Definition \ref{def.renor-parawake}). Now consider two cases:
\begin{itemize}
    \item if $\bigcap_{n=0}^\infty \partial{\cP^n_{i_n}} = \emptyset$, then any $b\in\mathcal{L}$ either satisfies Assumption \ref{assum.well-defined}, or is separable. If we are in the second situation, then by Proposition \ref{para separable}, $b$ is a parabolic parameter. Notice that parabolic parameters are discrete. Hence in both situations, if $\cL$ does not reduce to one point, then there will be at least two distinct parameters $b_1,b_2\in\cL$ satisfying Assumption \ref{assum.well-defined}. By Proposition \ref{match puzzle}, $f_{b_1}$ and $f_{b_2}$ are combinatorially equivalent. By Theorem \ref{thm.rigidity}, $b_1 = b_2$, which leads to a contradiction.
    \item if $\bigcap_{n=0}^\infty \partial{\cP^n_{i_n}} \not= \emptyset$, then by Theorem \ref{paratree}, this intersection contains exactly one point. Suppose by contradiction that $\cL$ is not a single point, then the same argument follows as in the above case.
\end{itemize}
It follows that the correspondence via $\Psi$ given in \thmref{paratree} extends homeomorphically to the closures of $\cT_\infty(\theta) / [b \sim -b]$ and $T(\q)\setminus D(\q)$.
\end{proof}

\begin{proof}[Proof of Corollary \ref{cor.main2}]
Recall the central locus $\mathcal{K}_\infty(\theta)$ defined in (\ref{eq.central_locus}). One can verify easily by definition that all parameters in $\mathbf{H}^3$ are non renormalizable, and hence by definition of $\mathcal{K}_\infty(\theta)$, we have $\overline{\bfH^3}\cap[\mathcal{P}^{cm}(\theta)]\subset[\mathcal{K}_\infty(\theta)]$. For the inverse inclusion, take any $[f]\in[\mathcal{K}_\infty(\theta)]$. We may suppose that $f\in\overline{\mathcal{U}_\infty}$. Thus by Theorem \ref{thm.main.siegel}, $f$ is accumulated by a sequence of capture components in ${\mathcal{K}}_\infty(\theta)$. On the other hand, by \cite[Thm. A]{BlChOvTi}, any capture component $\mathcal{B}$ in ${\mathcal{K}}_\infty(\theta)$ satisfies $[\mathcal{B}]\subset\partial\bfH^3$. This finishes the proof.
\end{proof}

\end{document}